\documentclass[12pt]{article}
\usepackage{amsfonts}
\usepackage{amsthm}
\usepackage{url}
\usepackage{mathtools}
\usepackage[all]{xy}
\usepackage{amssymb}
\usepackage{color}
\usepackage{mathrsfs}
\usepackage{tikz}
\usepackage{hyperref}
\usetikzlibrary{calc,intersections,through,backgrounds}

\usepackage{pdfsync}
\newtheorem{thm}{Theorem}[section]
\newtheorem{prop}[thm]{Proposition}
\newtheorem{lemma}[thm]{Lemma}
\newtheorem{coro}[thm]{Corollary}
\newtheorem{definition}[thm]{Definition} 

\newtheorem{conjecture}[thm]{Conjecture}

\newcommand{\Z}{{\mathbb Z}} 
\newcommand{\Q}{{\mathbb Q}}

\newcommand{\Com}{{\mathbb C}}
\newcommand{\Qp}{{\mathbb {Q}_{p}}}
\newcommand{\A}{{\mathbb A}}
\newcommand{\Sp}{{\mathrm{Spec} \:}}

\newcommand{\cI}{{\mathrm{c\text{--} Ind}}}
\DeclareMathOperator{\Spm}{\mathrm{m-Spec}}

\makeatletter
\let\c@equation\c@thm
\makeatother
\numberwithin{equation}{section}

\title{On the Breuil--Schneider conjecture: Generic case}

\author{Alexandre Pyvovarov}

\date{\today}

\begin{document}

\maketitle

\begin{abstract}

We use the Taylor--Wiles--Kisin patching method to prove some new cases of the Breuil--Schneider conjecture.

\end{abstract}

\tableofcontents

\section{Introduction}\label{Intro}
The aim of this work is to deduce some new cases of Breuil--Schneider conjecture using the patching construction of \cite{MR3529394}. The conjecture in question predicts the existence of $GL_n(F)$-invariant norms on locally algebraic $GL_n(F)$-representations, where $F$ is a $p$-adic field. This conjecture was first proposed by Breuil and Schneider in \cite{MR2359853}, and it may be seen as first evidence to a $p$-Langlands correspondence. For a brief survey of this conjecture one may consult \cite{MR3409331}. This introduction is strongly influenced by this survey. Our main result Theorem \ref{t2}, computes the locally algebraic vectors in the unitary Banach space representation obtained from this theory.

\subsection {Notation} \label{sec: Not}

Let $p$ a prime number such that $p\nmid 2n$. Let $F$ be a finite extension of $\Qp$ with a finite residue field $k_{F}$. Let $\mathcal{O}_F$ be its complete discrete valuation ring, let $\mathfrak{p}$ be the maximal ideal of $\mathcal{O}_F$ with uniformizer $\varpi$, and let $q=|\mathcal{O}_F/\varpi \mathcal{O}_F|$. Let $G=GL_n(F)$.

Let $E$ be a finite extension of $\Qp$ (the field of coefficients), $\mathcal{O}$ the ring of integers of $E$ and $\mathbb{F}$ the residue field. Fix a residual Galois representation  $\overline{r} : G_F \longrightarrow GL_n(\mathbb{F})$ of the local Galois group $G_F:= \mathrm{Gal}(\bar{F}/F)$. We assume that $E$ is large enough to contain all the embeddings $F \hookrightarrow \overline{\Q}_{p}$.

Fix an isomorphism $\Com \simeq \overline{\Q}_{p}$. In the literature, smooth representations of $GL_n(F)$ are studied on complex vector spaces, hence the coefficients are in $\overline{\Q}_{p}$.  The section 3.13 \cite{MR3529394} provides a framework which shows that those results, which are valid in the classical theory over complex numbers, have an analogue over $E$ by faithfully flat descent from $\overline{\Q}_{p}$ to $E$. For instance if $\pi$ is an irreducible representation of $GL_n(F)$ with coefficients in $E$, after extending scalars to larger extension $E'/E$ we may assume that $\pi$ is absolutely irreducible and hence do the base change to $\overline{\Q}_{p}$. In this way all the results stated with $E$-coefficients can be proven over $\overline{\Q}_{p}$ without loss of generality. With this in mind we will cite results from \cite{Pyv1} and \cite{Pyv2}, which were proven for $\overline{\Q}_{p}$-coefficients and use them with $E$-coefficients where $E$ is a finite extension of $\Qp$.

In this paper we will follow the notation and conventions  of \cite{MR3529394}(cf. 1.8), unless otherwise is stated.

\subsection{The Breuil--Schneider conjecture}\label{I.1}

For the background on the Breuil--Schneider conjecture we refer the reader to \cite[\S 2]{MR3409331}.

Now suppose $r : G_F \rightarrow GL_n(E)$ is a potentially semi-stable lift of $\overline{r}$, with  Hodge-Tate weights $\mathrm{HT}_{\kappa}= \{i_{\kappa,1} < \ldots < i_{\kappa,n}\}$, for each embedding $\kappa : F\hookrightarrow E$. By Fontaine's recipe one associates an $n$-dimensional Weil-Deligne representation $WD(r)$ to $r$ with coefficients in $\overline{\Q}_{p} \simeq \Com$. Let $\mathrm{rec}$ denote the classical local Langlands correspondence with coefficients in $\Com$ normalized in a way such that the central character of an irreducible smooth representation of $GL_n(F)$ corresponds to the determinant of the associated Weil-Deligne representation via the local class field theory. This is compatible with convention in the book \cite{MR1876802}, see Lemma VII.2.6. Let $\mathrm{rec}_p$ denote the local Langlands correspondence over $\overline{\Q}_{p}$, defined by $\iota \circ \mathrm{rec}_p = \mathrm{rec}\circ \iota$ and define $r_p(\pi) = \mathrm{rec}_p(\pi \otimes |\det|)$. Let $\pi_{sm}(r)$ an irreducible smooth representation of $GL_n(F)$ with coefficients in $E$ defined by $\pi_{sm}(r) = r_p^{-1}(WD(r)^{F-ss})$, where $F-ss$ denotes the Frobenius semi-simplification of $WD(r)$. Assume that $\pi_{sm}(r)$ is generic, i.e. admits a Whittaker model. Then there is a model of $\pi_{sm}(r)$ with coefficients in $E$, denoted again $\pi_{sm}(r)$ which a smooth irreducible $E$-representation of $GL_n(F)$. We say that $r$ is generic when $\pi_{sm}(r)$ is generic. In the case when $\pi_{sm}(r)$ is not generic, we need to do some modifications, see \cite{MR2359853} for more details. Indeed, by the Bernstein-Zelevinsky classification, $\pi_{sm}(r)$ is a Langlands quotient and there is a unique parabolic induction, denoted $\pi_{gen}(r)$, such that $ \pi_{gen}(r)\twoheadrightarrow \pi_{sm}(r)$. This representation has a model over $E$, which we will denote again by $\pi_{gen}(r)$.

To the multi-set $\{ \mathrm{HT}_{\kappa} \}_{\kappa : F \hookrightarrow E}$ one can attach an irreducible algebraic representation of $\mathrm{Res}_{F/\Qp}(GL_n/F)$, which we evaluate at $E$ to get an algebraic representation $\pi_{alg}(r)$ of $GL_n(F)$. More precisely, for each $\kappa$, let $\pi_{alg,\kappa}(r)$ be the irreducible algebraic representation of $GL_n(F)$ of highest weight $\{-i_{\kappa,1},\ldots,-i_{\kappa,n}+n-1\}$ relative to the upper-triangular Borel. Then define $\pi_{alg}(r) = \bigotimes_{\kappa }\pi_{alg,\kappa}(r)$, with $GL_n(F)$ acting diagonally. This is the representation $L_{\xi} \otimes_{\mathcal{O}}E$ with notation of section 1.8\cite{MR3529394}, with $\xi_{\kappa,j}= -i_{\kappa,j}+j-1$.

\medskip

Define: $BS(r):= \pi_{gen}(r)\otimes_{E} \pi_{alg}(r)$.

\medskip
The conjecture, which we state in the generic case, predicts that irreducible locally algebraic representations of $G$ admit integral structures if and only if they are related to Galois representations. More precisely:

\begin{conjecture}
Let $\pi$ be an absolutely irreducible generic representation of $GL_n(F)$ and $\sigma$ an irreducible algebraic representation of algebraic group $\mathrm{Res}_{F/\Qp} GL_n / F$, both having coefficients in $E$. Then the following statements are equivalent:

\begin{enumerate}
\item[(1)] $\pi\otimes_{E} \sigma$ admits a $G$-invariant norm.
\item[(2)] There is a potentially semi-stable Galois representation $r : G_F \rightarrow GL_n(E)$ such that $\pi=\pi_{sm}(r)$ and $\sigma=\pi_{alg}(r)$.
\end{enumerate}

\end{conjecture}

The implication (1) $\Rightarrow$ (2), was proven by Hu in full generality in his paper \cite{MR2560407}. The converse is still largely open. In \cite{MR3529394}, the authors prove many cases of this conjecture by constructing an admissible unitary $E$-Banach space representation $V(r)$ of $GL_n(F)$, such that the locally algebraic vectors in $V(r)$ are isomorphic to $BS(r)$ as $GL_n(F)$. In \cite{MR3529394} the authors assume that $r$ is potentially crystalline. This corresponds to the case when the monodromy operator $N$ of the Weil-Deligne representation $WD(r)$ is zero. 

In this paper we extend the methods of \cite{MR3529394} to handle the case, when $r$ is potentially semi-stable.  This corresponds to the case when the monodromy operator $N$ of the Weil-Deligne representation $WD(r)$ is allowed to be arbitrary. We will be mostly concerned with the case when the Galois representation $r$ is generic, in this case $BS(r)=\pi_{sm}(r)\otimes_{E}\pi_{alg}(r)$ is irreducible. We prove, that the locally algebraic vectors of $V(r)$ are isomorphic to $BS(r)$. This will allow us to deduce new cases for the implication (2) $\Rightarrow$ (1).

It was observed in \cite{MR3409331}, that norms are related to lattices, up to equivalence. Indeed given a norm $\|.\|$, the lattice $\Lambda$ will be a unit ball for this norm. Conversely, given a lattice $\Lambda$, set  $\|x\| = q_E^{-v_{\Lambda}(x)}$, where $q_E = |\mathbb{F}|$ and $v_{\Lambda}(x)$ is the largest $k$ such that $x \in \varpi_E^{k}\Lambda$ ($\varpi_E$ is a uniformizer of $E$). Thus we are looking for integral structures in $BS(r)$.

\subsection{Typical representations}\label{I.2}

Recall that the Bernstein decomposition expresses the category of smooth $\overline{\Q}_p$-valued representations of a $p$-adic reductive group $G$ as the product of certain indecomposable full subcategories, called Bernstein components. Those components are parametrized by the inertial classes. Let us now recall the definition of an inertial class. Let $M$ be a Levi subgroup of some parabolic subgroup of $G$ and let $\rho$ be an irreducible supercuspidal representation of $M$ and consider a set of pairs $(M, \rho)$ as above. We say that two pairs $(M_1, \rho_1)$ and $(M_2, \rho_2)$ are inertially equivalent if and only if there is $g \in G$ and an unramified character $\chi$ of $M_2$ such that:

\begin{center}
$M_2=M_1^{g}$ and $\rho_2 \simeq \rho_1^g \otimes\chi$
\end{center} 

\noindent where $M_1^g:=g^{-1}M_1g$ and $\rho_1^g(x) = \rho_1(gxg^{-1})$, $\forall x \in M_1^g$. An equivalence class of all such pairs will be denoted $[M,\rho]_{G}$. The set of inertial class equivalences of all such pairs will be denoted by $\mathcal{B}(G)$.

Let  $\mathcal{R}(G)$ be the category of all smooth  $E$-representations of $G$. We denote by $i_{P}^{G} : \mathcal{R}(M) \longrightarrow \mathcal{R}(G)$ the normalized parabolic induction functor, where $P=MN$ is a parabolic subgroup of $G$ with Levi subgroup $M$. Let $\overline{P}$ be the opposite parabolic with respect to $M$. We use the notation $\mathrm{Ind}$ and $\cI$ to denote the induction and compact induction respectively.

\medskip

\noindent We are given an inertial class $\Omega:=[M,\rho]_{G}$, where $\rho$ is a supercuspidal representation of $M$ and $D:=[M,\rho]_{M}$. To any inertial class $\Omega$ we may associate a full subcategory $\mathcal{R}^{\Omega}(G)$ of $\mathcal{R}(G)$, such that $(\pi,V) \in \mathrm{Ob}(\mathcal{R}^{\Omega}(G))$ if and only if every irreducible $G$-subquotient $\pi_{0}$ of $\pi$ appear as a composition factor of $i_{P}^{G}(\rho \otimes \omega)$ for $\omega$ some unramified character of $M$ and $P$ some parabolic subgroup of $G$ with Levi factor $M$. The category $\mathcal{R}^{\Omega}(G)$ is called a Bernstein component of $\mathcal{R}(G)$. We will say that a representation $\pi$ is in $\Omega$ if $\pi$ is an object of $\mathcal{R}^{\Omega}(G)$. According to \cite{MR771671}, we have a decomposition:

\[\mathcal{R}(G) = \prod_{\Omega \in \mathcal{B}(G)} \mathcal{R}^{\Omega}(G)\]

So in order to understand the category $\mathcal{R}(G)$, it is enough to restrict our attention to the components. We may understand those components via the theory of types. This is a way to parametrize all the irreducible representations of $G$ up to inertial equivalence using irreducible representations of compact open subgroups of $G$.

Let $J$ be a compact open subgroup of $G$ and let $\lambda$ be an irreducible representation of $J$. We say that $(J, \lambda)$ is an $\Omega$-type if and only if for every irreducible representation $(\pi,V) \in \mathrm{Ob}(\mathcal{R}^{\Omega}(G))$, $V$ is generated as a $G$-representation by  the $\lambda$-isotypical component of $V$. 

Let $\mathcal{R}_{\lambda}(G)$ be a full subcategory of $\mathcal{R}(G)$ such that $(\pi,V) \in \mathrm{Ob}(\mathcal{R}_{\lambda}(G))$ if and only if $V$ is generated as a $G$-representation by $V^{\lambda}$ (the $\lambda$-isotypical component of $V$). Since  $(J,\lambda)$ is an $\Omega$-type, it follows from \cite[(3.6) and (3.5)]{MR1643417} that $\mathcal{R}^{\Omega}(G)=\mathcal{R}_{\lambda}(G)$, because $\lambda$ defines a special idempotent.

Let $K$ be a maximal compact open subgroup of $G$ containing $J$. We say that an irreducible representation $\sigma$ of $K$ is \textit{typical for $\Omega$} if for any irreducible representation $\pi$ of $G$, $\mathrm{Hom}_{K}(\sigma,\pi)\neq 0$ implies that $\pi$ is an object in $\mathcal{R}^{\Omega}(G)$.

\bigskip

\noindent Define $\mathcal{H}(G,\lambda):= \mathrm{End}_{G}(\cI_J^G \lambda)$. Then for any $\Omega$-type $(J,\lambda)$, by Theorem 4.2 (ii)\cite{MR1643417},  the functor:
$$\begin{array}{ccccc}
\mathfrak{M}_{\lambda} & : & \mathcal{R}_{\lambda}(G) & \to & \mathcal{H}(G,\lambda)-Mod \\
 & & \pi &  \mapsto & \mathrm{Hom}_{J}(\lambda, \pi) = \mathrm{Hom}_{G}(\cI_J^G \lambda, \pi)\\
\end{array}$$

\noindent induces an equivalence of categories. 

Write $\mathfrak{Z}_{\Omega}$ for the centre of category $\mathcal{R}^{\Omega}(G)$ and $\mathfrak{Z}_{D}$ for the centre of category $\mathcal{R}^{D}(M)$, which is defined in the same way as $\mathcal{R}^{\Omega}(G)$. Recall that the centre of a category is the ring of endomorphisms of the identity functor. For example the centre of the category $\mathcal{H}(G,\lambda)-Mod$ is $Z(\mathcal{H}(G,\lambda))$, where $Z(\mathcal{H}(G,\lambda))$ is the centre of the ring $\mathcal{H}(G,\lambda)$. We will call $\mathfrak{Z}_{\Omega}$ a Bernstein centre.

\medskip

For $G=GL_n(F)$, the types can be constructed in an explicit manner (cf. \cite{MR1204652}, \cite{MR1643417} and \cite{MR1711578}) for every Bernstein component. Moreover, Bushnell and Kutzko have shown that $\mathcal{H}(G, \lambda)$ is naturally isomorphic to a tensor product of affine Hecke algebras of type A.

\medskip

The simplest example of a type is $(I, 1)$, where $I$ is Iwahori subgroup of $G$ and $1$ is the trivial representation of $I$. In this case $\Omega = [T,1]_G$, where $T$ is the subgroup of diagonal matrices. We will refer to this example as the Iwahori case.

\medskip

In \cite{MR1728541} section 6 (just above proposition 2) the authors define irreducible $K$-representations $\sigma_{\mathcal{P}}(\lambda)$, where $\mathcal{P}$ is  a partition-valued function with finite support (cf. section 2 \cite{MR1728541}). There is a natural partial ordering, as defined in \cite{MR1728541}, on the partition-valued functions. Let $\mathcal{P}_{max}$ be the maximal partition-valued function and let $\mathcal{P}_{min}$ the minimal one. Define $\sigma_{max}(\lambda):=\sigma_{\mathcal{P}_{max}}(\lambda)$ and $\sigma_{min}(\lambda):=\sigma_{\mathcal{P}_{min}}(\lambda)$. Both $\sigma_{max}(\lambda)$ and $\sigma_{min}(\lambda)$ occur in $\mathrm{Ind}_{J}^{K} \lambda$ with multiplicity 1.

\medskip

In the Iwahori case, $\sigma_{min}(\lambda)$ is the inflation of Steinberg representation of $GL_n(k_F)$ to $K$ and $\sigma_{max}(\lambda)$ is the trivial representation.

\subsection{Main result} \label{I.3}

Let $\mathbf{v}  = \{ \mathrm{HT}_{\kappa} \}_{\kappa : F \hookrightarrow E}$ be a multiset of all Hodge-Tate weights and let $\tau : I_F \rightarrow GL_n(E)$ be an inertial type, i.e. $\tau$ is a representation of $I_F$ with open kernel which extends to a representation of the Weil group $W_F$ of $F$, where $I_F$ is the inertia subgroup of $G_F$. We let $R_{\tilde{\mathfrak{p}}}^{\square}$ denote  the universal $\mathcal{O}$-lifting ring of $\overline{r}$. Then there is a ring $R_{\tilde{\mathfrak{p}}}^{\square}(\sigma_{min}):= R^{\square}_{\overline{r}}(\tau, \mathbf{v})$, which is the unique reduced and $p$-torsion free quotient of $R_{\tilde{\mathfrak{p}}}^{\square}$ corresponding to potentially semi-stable lifts of weight $\sigma_{alg}$ (i.e. of weight $\mathbf{v}$) and inertial type $\tau$. This ring was constructed in \cite{MR2373358}. Moreover there is a \textquotedblleft universal admissible filtered $(\varphi,N)$-module" $D^{\square}_{\overline{r}}(\tau, \mathbf{v})$ which is a locally free $R^{\square}_{\overline{r}}(\tau, \mathbf{v})[1/p]\otimes_{\Qp}F_{0}$-module of rank $n$, where $F_0=W(k_F)[1/p]$, and $W(k_F)$ are the Witt vectors of the residue field of $F$. The module $D^{\square}_{\overline{r}}(\tau, \mathbf{v})$ comes equipped with a universal Frobenius, denoted by $\varphi$.

Let $\sigma_{alg}$ the restriction to $K$ of $\pi_{alg}(r)$. Define $\sigma_{min}:= \sigma_{min}(\lambda)\otimes \sigma_{alg}$ and $\mathcal{H}(\sigma_{min}):= \mathrm{End}_{G}(\mathrm{c\text{--} Ind}_{K}^{G} \sigma_{min})$.

We have fixed a type $\tau$, so there is a finite extension $L$ of $F$ such that the restriction of every Galois representation $r_x$ to $G_L$ is semi-stable. Let $L_0$ its maximal unramified subfield of $L$. We assume $[L_0: \Qp] = |\mathrm{Hom}_{\Qp}(L_0,E)|$ and we let $p^{f}$ be the cardinality of the residue field of $L_0$. By universal we mean that the specialization $D_x$ of $D^{\square}_{\overline{r}}(\tau, \mathbf{v})$ at the closed point $x$ of $R_{\tilde{\mathfrak{p}}}^{\square}(\sigma_{min})[1/p]$ with residue field $E_x$ is an admissible filtered $(\varphi,N, \mathrm{Gal}(L/F))$-module attached to the Galois representation $r_x$ given by the point $x$. 

On the other hand one may show using the classical local Langlands correspondence that $\tau$ determines a Bernstein component $\mathcal{R}^{\Omega}(G)$. We prove that there is a map that interpolates the local Langlands correspondence, more precisely:

\begin{thm}
There is an $E$-algebra homomorphism
\[\beta: \mathcal{H}(\sigma_{min}) \longrightarrow R_{\tilde{\mathfrak{p}}}^{\square}(\sigma_{min})[1/p]\]
\noindent such that for any closed point $x$ of $R_{\tilde{\mathfrak{p}}}^{\square}(\sigma_{min})[1/p]$ with residue field $E_x$, the action of $\mathfrak{Z}_{\Omega}$ on the smooth $G$-representation $\pi_{sm}(r_x)$ factors as $\beta$ composed with the evaluation map  $R_{\tilde{\mathfrak{p}}}^{\square}(\sigma_{min})[1/p] \longrightarrow E_x$.
\end{thm}

This result generalizes Theorem 4.1 \cite{MR3529394} (i.e. if we restrict to the crystalline locus the two maps coincide), however the proof does not follow methods of this paper. Instead we give an explicit construction of this map.

We will sketch the construction of $\beta$ in the Iwahori case, in this case $L=F$ because the lifts we consider are semi-stable.  By Satake isomorphism and Corollary 7.2 \cite{Pyv1}, we have $\mathcal{H}(\sigma_{\min}) \simeq E[\theta_1,\ldots,\theta_{n-1},(\theta_{n})^{\pm 1}]$, where $\theta_r$ is a double coset operator $ \left[K \begin{pmatrix}
 \varpi I_{r}& 0 \\ 
0 & I_{n-r}
\end{pmatrix} K \right] $. All the details of this construction will be discussed in the section \ref{LA.2.2}.

If for an embedding $\kappa$ the Hodge-Tate weights are $i_{\kappa,1}<\ldots<i_{\kappa,n}$ define $\xi_{\kappa,j}=-i_{\kappa,j} + (j-1)$. Then the map $\beta:\mathcal{H}(\sigma_{min}) \longrightarrow R^{\square}_{\overline{r}}(\tau, \mathbf{v})[1/p]$ is given by the assignment 
$$\theta_r \mapsto \varpi^{-\sum_{\kappa}\sum\limits_{i=r}^{n}\xi_{\kappa,j}} q^{\frac{r(1-r)}{2}}\mathrm{Tr}(\bigwedge^{r} \varphi^f)$$

For $\mathcal{P}$ any partition-valued function, define $\sigma_{\mathcal{P}} := \sigma_{\mathcal{P}}(\lambda) \otimes \sigma_{alg}$,  where $\sigma_{\mathcal{P}}(\lambda)$ was defined above and $\sigma_{alg}$ is the restriction to $K$ of an irreducible algebraic representation of $\mathrm{Res}_{F/\Qp}GL_n$ given by the Hodge-Tate weights.  Fix a $K$-stable $\mathcal{O}$-lattice $\sigma_{\mathcal{P}}^{\circ}$ in $\sigma_{\mathcal{P}}$. Set
\[M_{\infty}(\sigma_{\mathcal{P}}^{\circ}):= \left( \mathrm{Hom}_{\mathcal{O}[[K]]}^{cont}(M_{\infty}, (\sigma_{\mathcal{P}}^{\circ})^{d})\right) ^{d}\]
\noindent where $M_{\infty}$ is $R_{\infty}$-module constructed in section 2 \cite{MR3529394} by patching process and $(.)^d = \mathrm{Hom}_{\mathcal{O}}^{cont}(., \mathcal{O})$ denotes the Shikhof dual. Since $\sigma_{\mathcal{P}}^{\circ}$ is a free $\mathcal{O}$-module of finite rank, it follows from the proof of Theorem 1.2 of \cite{MR1900706} that Schikhof duality induces an isomorphism
\[\mathrm{Hom}_{\mathcal{O}[[K]]}^{cont}(M_{\infty}, (\sigma_{\mathcal{P}}^{\circ})^{d}) \simeq \mathrm{Hom}_{K}(\sigma_{\mathcal{P}}^{\circ},(M_{\infty})^{d})\]

\noindent and Frobenius reciprocity gives $$\mathrm{Hom}_{K}(\sigma_{\mathcal{P}}^{\circ},(M_{\infty})^{d})\simeq \mathrm{Hom}_{G}(\mathrm{c\text{--} Ind}_{K}^{G} \sigma_{\mathcal{P}}^{\circ},(M_{\infty})^{d}).$$
The action of $\mathfrak{Z}_{\Omega}$ on $\mathrm{c\text{--} Ind}_{K}^{G} \sigma_{\mathcal{P}}$ induces an action on $M_{\infty}(\sigma_{\mathcal{P}}^{\circ})[1/p]$.

To any closed point of $x \in \Spm R^{\square}_{\tilde{\mathfrak{p}}}(\sigma_{min})[1/p]$ we can attach a partition-valued function $\mathcal{P}_x$, which encodes information about the shape of monodromy operator of the admissible filtered $(\varphi,N)$-module $D_x$. We prove that there is a reduced $p$-torsion free quotient $R^{\square}_{\tilde{\mathfrak{p}}}(\sigma_{\mathcal{P}})$ of $R^{\square}_{\tilde{\mathfrak{p}}}(\sigma_{min})$, such that $x \in \Spm R^{\square}_{\tilde{\mathfrak{p}}}(\sigma_{\mathcal{P}})[1/p]$ if and only if $\mathcal{P}_x \geq \mathcal{P}$. When $\sigma_{\mathcal{P}}= \sigma_{min}$, the ring corresponds to all the potentially semi-stable lifts and this is compatible with the notation introduced at the beginning. The other extreme case is $R^{\square}_{\tilde{\mathfrak{p}}}(\sigma_{max})$; this ring parametrizes all the potentially crystalline lifts.

As a part of patching construction we know that $R_{\infty}$ is an $R_{\tilde{\mathfrak{p}}}^{\square}$-algebra. We define $R_{\infty}(\sigma_{\mathcal{P}})':=R_{\infty}\otimes_{R_{\tilde{\mathfrak{p}}}^{\square}}R_{\tilde{\mathfrak{p}}}^{\square}(\sigma_{\mathcal{P}})$. Let $R_{\infty}(\sigma_{\mathcal{P}})$ be the quotient of $R_\infty$ which acts faithfully on $M_{\infty}(\sigma_{\mathcal{P}}^{\circ})$. The usual commutative algebra arguments underlying the Taylor-Wiles-Kisin method, as in \cite{MR3529394}, show that $M_{\infty}(\sigma_{\mathcal{P}}^{\circ})$ is a maximal Cohen-Macaulay module over $R_{\infty}(\sigma_{\mathcal{P}})$. Moreover we prove an important result about the support of $M_{\infty}(\sigma_{min}^{\circ})$:

\begin{prop}\label{p2}
\begin{enumerate}
\item The module $M_{\infty}(\sigma_{min}^{\circ})[1/p]$ is locally free of rank one over the regular locus of $ \Sp R_{\infty}(\sigma_{min})[1/p]$.
\item $\Sp R_{\infty}(\sigma_{min}) [1/p]$ is a union of irreducible components of \\$\Sp R_{\infty}(\sigma_{min})'[1/p]$.
\end{enumerate}
\end{prop}

The components appearing in the second statement of the Proposition \ref{p2} are termed \textit{automorphic components}. The proof of the proposition above is similar to Lemma 4.18 \cite{MR3529394}. The action of $\mathfrak{Z}_{\Omega}$ on $M_{\infty}(\sigma_{min}^{\circ})[1/p]$ induces an $E$-algebra homomorphism: 
$$\alpha : \mathfrak{Z}_{\Omega} \longrightarrow \mathrm{End}_{R_{\infty}[1/p]}(M_{\infty}(\sigma_{min}^{\circ})[1/p])$$
\noindent From the Proposition \ref{p2}, we deduce that:

\begin{thm}
We have the following commutative diagram:
$$\xymatrix{
(\Sp R_{\infty}(\sigma_{min}) [1/p])^{reg} \ar@{^{(}->}[d] \ar[r]^-{\alpha^{\sharp}} & \Sp \mathcal{H}(\sigma_{min}) \\
\Sp R_{\infty}(\sigma_{min}) [1/p] \ar[r]^{can} &\Sp R_{\tilde{\mathfrak{p}}}^{\square}(\sigma_{min})[1/p], \ar[u] }$$

\noindent where $(\Sp R_{\infty}(\sigma_{min}) [1/p])^{reg}$ is the regular locus of $\Sp R_{\infty}(\sigma_{min}) [1/p]$ and $\alpha^{\sharp}$ the map induced by $\alpha$.
\end{thm}

Just as in \S 4.28 \cite{MR3529394}, the main technique is to convert information on locally algebraic vectors in the completed cohomology into commutative algebra statements about the module $M_{\infty}(\sigma_{min}^{\circ})$ using results on $K$-typical representations that we have explained in the previous section.

Let $x$ be a closed $E$-valued point of $\Sp R_{\infty} (\sigma_{min})[1/p]$. The corresponding Galois representation $r_x$ is given by the homomorphism $x : R_{\tilde{\mathfrak{p}}}^{\square} \rightarrow \mathcal{O}$, which we extend arbitrarily to homomorphism $x: R_\infty \rightarrow \mathcal{O}$. Then 

\begin{equation}\label{V}
V(r_{x}):=\mathrm{Hom}_{\mathcal{O}}^{cont}(M_{\infty} \otimes_{R_{\infty},x}\mathcal{O},E)
\end{equation}

\noindent is an admissible unitary $E$-Banach space representation of $G$. The main result is the following theorem:

\begin{thm}\label{t2}
Let $x$ be a closed $E$-valued point of $\Sp R_{\infty} (\sigma_{min})[1/p]$, such that  $\pi_{sm}(r_{x})$ is generic and irreducible. Then
$$V(r_{x})^{l.alg} \simeq \pi_{sm}(r_x) \otimes \pi_{alg}(r_x),$$
where $(.)^{l.alg}$ denotes the subspace of locally algebraic vectors. 
\end{thm}

Since  the action of $G$ on $V(r_x)$ is unitary, we obtain:

\begin{coro}
Suppose $p\nmid 2n$, and that $r : G_F \longrightarrow GL_{n}(E)$ is a generic potentially semi-stable Galois representation of regular weight. If $r$ correspond to a closed point $x \in \Sp R_{\infty}(\sigma_{min})[1/p]$, then $\pi_{sm}(r)\otimes \pi_{alg}(r)$ admits a non-zero unitary admissible Banach completion.
\end{coro}

It is conjectured in \cite{MR3529394} that $V(r_x)$ depends only on the Galois representation $r_x$ and that $r_x \mapsto V(r_x)$ realizes the hypothetical $p$-adic local Langlands correspondence. Our Theorem \ref{t2} provides further evidence of this conjecture.

Now, we will recall some more notation from \cite{MR3529394}. The globalization of $\overline{r}$ constructed in section 2.1 \cite{MR3529394}, provides us with a global imaginary CM field $\tilde{F}$ with maximal totally real subfield $\tilde{F}^{+}$. We refer the reader to section 2.1 \cite{MR3529394}, for more details and precise definitions. Let $S_p$ denote the set of primes of $\tilde{F}^{+}$ dividing $p$. Fix $\mathfrak{p}|p$. For each $v \in S_p$, let $\tilde{v}$ be a choice of a place in $\tilde{F}$ lying above $v$, as defined in section 2.4 \cite{MR3529394}.

The main result of this article is to compute the locally algebraic vectors for $V(r_{x})$, a candidate for the $p$-adic local Langlands correspondence, at the smooth points which lie on some automorphic component. 

\subsection{Outline of the paper}

This article is organised as follows: After recalling few definitions in section \ref{LA.1}, we will construct the map $\mathcal{H}(\sigma_{min}) \longrightarrow R_{\tilde{\mathfrak{p}}}^{\square}(\sigma_{min})[1/p]$, which interpolates the local Langlands correspondence in section \ref{LA.2}. Then in section \ref{LA.3} we will introduce a stratification of $R_{\tilde{\mathfrak{p}}}^{\square}(\sigma_{min})$ with respect to the partition-valued function, which will help us to study the support of $M_{\infty}(\sigma_{min}^{\circ})$. The goal of the section \ref{LA.4} will be to prove that the action of $\mathcal{H}(\sigma_{min})$ on $M_{\infty}(\sigma_{min}^{\circ})$ is compatible with the interpolation map constructed in section \ref{LA.2}. In order to deal with the monodromy of potentially semi-stable Galois representations we will use results from \cite{Pyv2}.  This will be stated in more precise manner as Theorem \ref{4.10} and the results about the support of $M_{\infty}(\sigma_{min}^{\circ})$ will be given in section \ref{LA.5}. In the section \ref{LA.6}, the we will compute locally algebraic vectors using a global point where we know the result already. The main result of that section is the Theorem \ref{4.31}. Finally in the section \ref{A.1}, we will apply this theorem to deduce new cases of the Breuil--Schneider conjecture.

\section{Locally algebraic vectors. Definition. First properties}\label{LA.1}

In this section we reproduced some parts of the appendix in \cite{MR1835001}. Let $E/\Qp$ be a finite extension and $V$ a vector space over $E$ and $V$ is a $G$-representation. We begin with a definition, a vector $v\in V$ is termed \textbf{locally algebraic} if:
\medskip

The  orbit map of the vector $v$ is locally algebraic, i.e. for $v \in V$, there is a compact open subgroup $K_v$ in $G$, and a finite dimensional subspace $U$ of $V$ containing the vector $v$ such that $K_v$ leaves $U$ invariant and operates on $U$ via the restriction to $K_v$ of a finite dimensional algebraic representation of the algebraic group scheme $\mathrm{Res}_{F/\Qp}GL_n$. 
\medskip

Similarly a representation $\pi$ of $G=GL_n(F)$ on $V$ is called \textbf{locally algebraic} if  any vector $v \in V$ is locally algebraic. According to \cite[Theorem 1]{MR1835001} every irreducible locally algebraic representation of $G$ is of the form $\pi_{sm} \otimes \pi_{alg}$, where $\pi_{sm}$ is a smooth irreducible representation of $G$ and $\pi_{alg}$ is an irreducible algebraic representation of $G$. For any Banach vector space representation $V$ of $G$ we have the following functor $V \mapsto V^{l.alg}$, where $V^{l.alg}$ is the subspace of locally algebraic vectors in $V$.

\medskip

\textbf{Notation.} Let $\pi$ be an irreducible representation of $G$, then we will write $\pi^{l.alg} = \pi_{sm}\otimes \pi_{alg}$, where $\pi_{sm}$ is a smooth irreducible representation of $G$ and $\pi_{alg}$ is an irreducible algebraic representation of $G$.

\section{Interpolation map}\label{LA.2}

In this section we will construct an analogue of the map $\eta$ from Theorem 4.1 \cite{MR3529394} in the potentially semi-stable case. First we extend a few results from section 3 of \cite{MR3529394}. Let $\pi$ be any irreducible representation, then the action of $\mathfrak{Z}_{\Omega}$ on $\pi$ defines a $E$-algebra morphism $ \chi_{\pi} : \mathfrak{Z}_{\Omega} \longrightarrow \mathrm{End}_{G}(\pi)\simeq E$. In the next Lemma we introduce new notation which will be used in this paper.

\begin{lemma}\label{4.2}
Let $\pi$ be an irreducible generic representation of $G=GL_{n}(F)$. Then by Bernstein-Zelevinsky classification (cf. \cite{MR584084}) there are pairwise non-isomorphic supercuspidal representations $\pi_i$ ($1 \leq i \leq s$) and segments $\Delta_{i,j}=(\pi_i \otimes \chi_{i,j})\otimes \ldots \otimes (\pi_i \otimes \chi_{i,j}\otimes |\det|^{k_{i,j}-1})$ ($1 \leq j \leq r_i$), where $\chi_{i,j}$ are unramified characters and $k_{i,j}$ are positive integers, such that:
\[\pi \simeq Q(\Delta_{1,1})\times \ldots\times Q(\Delta_{1,r_1})\times\ldots \times Q(\Delta_{s,1})\times \ldots \times Q(\Delta_{s,r_s}),\]

\noindent where $Q$ denotes the Langlands quotient (cf. \cite[Section 1.2]{MR1265559}). Notice that all the segments $\Delta_{i,j}$ and $\Delta_{i',j'}$ are not linked for $i \neq i'$, this means that any permutation of blocs $Q(\Delta_{i,1})\times \ldots\times Q(\Delta_{i,r_i})$ gives an isomorphic representation.

Define $\tilde{\Delta}_{i,j}:= (\pi_i \otimes \chi_{i,j}\otimes |\det|^{1-k_{i,j}})\otimes \ldots \otimes (\pi_i \otimes \chi_{i,j})$ and consider it as a representation of a corresponding Levi subgroup. Write,
\[\eta := \tilde{\Delta}_{1,1}\times \ldots\times \tilde{\Delta}_{1,r_1}\times\ldots \times \tilde{\Delta}_{s,1}\times \ldots \times \tilde{\Delta}_{s,r_s}\]
Notice that any permutation of blocs $\tilde{\Delta}_{i,1}\times \ldots\times \tilde{\Delta}_{i,r_i}$ gives a representation isomorphic to $\eta$. Then we have:
\[\mathrm{c\text{--} Ind}_{K}^{G} \sigma_{max}(\lambda) \otimes_{\mathfrak{Z}_{\Omega},\chi_{\pi}} E \simeq \eta\]

\noindent Moreover the action of $\mathfrak{Z}_{\Omega}$ on $\eta$ is given by the maximal ideal $\chi_{\pi}$.
\end{lemma}

\begin{proof} The result follows by the argument similar to the one given in the proof of \cite[Corollary 3.11]{MR3529394}. Let  $\pi'$ be $G$-cosocle of $\eta$. Then by \cite[Proposition 3.10]{MR3529394}, we have $\mathrm{c\text{--} Ind}_{K}^{G} \sigma_{max}(\lambda) \otimes_{\mathfrak{Z}_{\Omega},\chi_{\pi'}} E \simeq \eta$.  However $\pi$ occurs as a subquotient of $\eta$ and the $G$-socle of $\eta$ is irreducible and occurs as a subquotient with multiplicity one. Since $\pi$ is generic, it follows that $\pi$ is the $G$-socle of $\eta$. Then the action of $\mathfrak{Z}_{\Omega}$ on $\eta$ factors through a maximal ideal, which is equal to $\chi_{\pi}$. By construction the action of $\mathfrak{Z}_{\Omega}$ on $\eta$ is given by $\chi_{\pi'}$, thus $\chi_{\pi'}=\chi_{\pi}$.
\end{proof}

We are given $\pi$ an irreducible generic representation as in lemma above.  We would like to describe $\chi_{\pi}$ in more concrete terms. In facts we would like to have a concrete description of the action of $\mathfrak{Z}_{\Omega}$ on $\pi \otimes |\det|^{\frac{n-1}{2}}$ in terms of eigenvalues of the associated Weil-Deligne representation by the local Langlands correspondence.

\medskip

Let $W:=W(k_{F})$ be the ring of Witt vectors of the residue field of $F$, recall that $\varpi$ is a uniformizer of $F$. Let $F_0 = W(k_F)[1/p]$, then $F/F_0$ is totally ramified. We will denote by $E(u) \in F_0[u]$ the Eisenstein polynomial of $\varpi$.

Let $R_{\tilde{\mathfrak{p}}}^{\square}(\sigma_{min}):=R^{\square}_{\overline{r}}(\tau, \mathbf{v})$ be the unique reduced and $p$-torsion
free quotient of $R_{\tilde{\mathfrak{p}}}^{\square}$ corresponding to potentially semi-stable lifts of weight $\sigma_{alg}$ (i.e. of weight $\mathbf{v}$) and inertial type $\tau$ constructed in \cite{MR2373358}.


It follows from the Theorem (2.5.5) (2) \cite{MR2373358} that $R^{\square}_{\overline{r}}(\tau, \mathbf{v})[1/p]$ is endowed with a universal $(\varphi,N)$-module $D^{\square}_{\overline{r}}(\tau, \mathbf{v})$, which is a locally free $R^{\square}_{\overline{r}}(\tau, \mathbf{v})[1/p]\otimes_{\Qp}F_{0}$-module of rank $n$.

Let us recall a few facts about $(\varphi,N)$-modules. The discussion about $(\varphi, N)$-modules follows \cite[section 4]{MR2359853}. We have two finite extensions $F$ (the base field) and $E$ (the coefficient field) of $\Qp$ such that $[F : \Qp] = |\mathrm{Hom}_{\Qp} (F, E)|$ where $\mathrm{Hom}_{\Qp}(F,E)$ denotes the set of all $\Qp$-linear embeddings of the field $F$ into the field $E$. We assume $F$ is contained in an algebraic closure $\overline{\Q}_{p}$ of $\Qp$. We denote by $q=p^{f_0}$ the cardinality of the residue field of $F$ and recall that $F_0$ is the maximal unramified subfield of $F$. If $e := [L : \Qp]/f_0$, we set $val_F(x) := e.val_{\Qp}(x)$ (where $val_{\Qp}(p) := 1$) and $|x|_F := q^{-val_F(x)}$ for any $x$ in a finite extension of $\Qp$. We denote by $W_F=W(\overline{\Q}_{p}/F)$ (resp. $G_F:=\mathrm{Gal}(\overline{\Q}_{p}/F)$) the Weil (resp. Galois) group of $F$ and by
$\mathrm{rec}_p : W(\overline{\Q}_{p}/F)^{ab} \to F^{\times}$ the reciprocity map sending the geometric Frobenius to the uniformizer.

Let $L$ be a finite Galois extension of $F$ and let $L_0$ be the maximal unramified subfield of $L$. We assume $[L_0: \Qp] = |\mathrm{Hom}_{\Qp}(L_0,E)|$ and we let $p^{f}$ be the cardinality of the residue field of $L_0$ and $\varphi_0$ be the Frobenius on $L_0$ (raising to the $p$ each component of the Witt vectors). Consider the following two categories:

\begin{enumerate}
\item  the category $\mathrm{WD}_{L/F}$ of representations $(r, N, V)$ of the Weil-Deligne group of $F$ on a $E$-vector space $V$ of finite dimension such that $r$ is unramified when restricted to $W(\overline{\Q}_{p}/L)$.

\item the category $\mathrm{MOD}_{L/F}$ of quadruples $(\varphi, N, \mathrm{Gal}(L/F), D)$ where $D$ is a free $L_0 \otimes_{\Qp}E$-module of finite rank endowed with a Frobenius $\varphi : D \to D$, which is $\varphi_0$-semi-linear bijective map, an $L_0 \otimes_{\Qp}E$-linear endomorphism $N : D \to D$ such that $N\varphi = p\varphi N$ and an action of $\mathrm{Gal}(L/F)$ commuting with $\varphi$ and $N$.
\end{enumerate}

There is a functor (due to Fontaine):
\[\mathrm{WD} : \mathrm{MOD}_{L/F} \longrightarrow \mathrm{WD}_{L/F}\]

The following proposition was proven in \cite{MR2359853}(Proposition 4.1):

\begin{prop}\label{2.1}
The functor $\mathrm{WD} : \mathrm{MOD}_{L/F} \longrightarrow \mathrm{WD}_{L/F}$ is an equivalence
of categories.
\end{prop} 

\noindent Denote by $\mathrm{MOD}$ a quasi inverse of the functor $\mathrm{WD}$.

\medskip

\noindent If $D$ is an object of $\mathrm{MOD}_{L/F}$, we define:
\[ t_N(D)= \frac{1}{[F:L_0]f} val_F(\mathrm{det}_{L_0}(\varphi^{f}|D))\]

\noindent For $\sigma : F \hookrightarrow E$, let $D_{L}=D \otimes_{L_0}L$ and :
\[D_{L,\sigma}= D_{L} \otimes_{L\otimes_{\Qp}E} (L\otimes_{F,\sigma}E)\]

Then one has $D_{L} \simeq \prod_{\sigma: F \to E}D_{L,\sigma}$. To give an $L\otimes_{\Qp}E$-submodule $\mathrm{Fil}^{i}D_{L}$ of $D_{L}$ preserved by $\mathrm{Gal}(L/F)$ is the same thing as to give a collection $(\mathrm{Fil}^{i}D_{L,\sigma})_{\sigma}$ where $\mathrm{Fil}^{i}D_{L,\sigma}$ is a free $L\otimes_{F,\sigma}E$-submodule of $D_{L,\sigma}$ (hence a direct factor as $L\otimes_{F,\sigma}E$-module) preserved by the action of $\mathrm{Gal}(L/F)$. If $(\mathrm{Fil}^{i}D_{L,\sigma})_{\sigma,i}$ is a decreasing exhaustive separated filtration on $D_{L}$ by $L\otimes_{\Qp}E$-submodules indexed by $i \in\Z$ and preserved by $\mathrm{Gal}(L/F)$, we define:
\[t_H(D_{L}) = \sum_{\sigma} \sum_{i \in \Z} i \dim_{L}(\mathrm{Fil}^{i}D_{L,\sigma}/\mathrm{Fil}^{i+1}D_{L,\sigma})\]

Recall that such a filtration is called admissible if $t_H(D_{L})=t_N(D)$ and if, for any $L_0$-vector subspace $D'\subseteq D$ preserved by $\varphi$ and $N$ with the induced filtration on $D'_{L}$, one has $t_H(D'_{L})\leq t_N(D')$.

\medskip

Our goal here is to construct a canonical map $\mathcal{H}(\sigma_{min}) \longrightarrow R^{\square}_{\overline{r}}(\tau, \mathbf{v})[1/p]$. We proceed in the following steps:
\begin{enumerate}
\item Take a smooth closed point $x \in \Sp R^{\square}_{\overline{r}}(\tau, \mathbf{v})[1/p]$.

\item The point $x$ is given by an $E$-algebra homomorphism  $x :R^{\square}_{\overline{r}}(\tau, \mathbf{v})[1/p] \longrightarrow E$ and it corresponds to $n$-dimensional Galois representation of $G_F$; denoted  $V_x$. Let $D_{st,L}(V_x):=(B_{st} \otimes_{\Qp} V_x)^{G_L}$, by construction this is also $D_x$, the specialization of $D^{\square}_{\overline{r}}(\tau, \mathbf{v})$ at closed point $x$.   The admissible filtered $(\varphi, N, \mathrm{Gal}(L/F))$-module $D_x=D_{st,L}(V_x)$ is equipped with Frobenius endomorphism $\phi_x$, which is the specialization of the universal Frobenius $\varphi$ on  $D^{\square}_{\overline{r}}(\tau, \mathbf{v})$ at $x$, i.e. $\varphi \otimes \kappa(x)=\phi_x$. 

\item Let $\tilde{\pi}$ be an irreducible representation of $GL_n(F)$ such that $$\mathrm{rec}_p(\tilde{\pi} \otimes |\det|^{\frac{1-n}{2}}) = WD(D_x),$$ \noindent here $WD(D_x)$ is the Weil-Deligne representation associated to $D_x$ via Fontaine's recipe. Let $\pi:=\tilde{\pi} \otimes |\det|^{\frac{1-n}{2}}$ so that  $\pi=\mathrm{rec}_p^{-1}(WD(D_x))$. 

\item Theorem 1.2.7 \cite{MR3546966} implies that the representation $\pi$ is generic, because $x$ is a smooth point.

\item Let $\eta:= \mathrm{c\text{--} Ind}_{K}^{G} \sigma_{max} \otimes_{\mathfrak{Z}_{\Omega},\chi_{\pi}}E$, as in Lemma \ref{4.2}. By the Lemma \ref{4.2} the action of $\mathfrak{Z}_{\Omega}$ on $\tilde{\pi}=\pi \otimes |\det|^{\frac{n-1}{2}}$ is identified with the action of $\mathfrak{Z}_{\Omega}$ on $\eta \otimes |\det|^{\frac{n-1}{2}}$. We will try to understand the action of $\mathfrak{Z}_{\Omega}$ on $\eta \otimes |\det|^{\frac{n-1}{2}}$.

\item We can interpret the action of $\mathfrak{Z}_{\Omega}$ on $\eta \otimes |\det|^{\frac{n-1}{2}}$ in terms of eigenvalues of the linearized canonical map obtained from the specialization of the absolute Frobenius $\varphi$ at point $x$. For this we use the decomposition of the spherical Hecke algebra of a semi-simple type as a tensor product of Iwahori Hecke algebras, and this decomposition restricts to $\mathfrak{Z}_{\Omega}$. Then we use the Satake isomorphism on each factor of $\mathfrak{Z}_{\Omega}$. In the Iwahori case there is just one factor in that tensor product decomposition.

\item From the previous step we can \textquotedblleft guess" a ring homomorphism $ \beta : \mathfrak{Z}_{\Omega} \longrightarrow R^{\square}_{\overline{r}}(\tau, \mathbf{v})[1/p]$. This map is canonical in the sense that if there was another map $\beta'$ it would coincide with $\beta$ on all of the smooth points, and since the smooth points are dense by Theorem 3.3.4 \cite{MR2373358}, the two maps have to be equal.

\item Finally $\mathcal{H}(\sigma_{min}(\lambda)) \to \mathcal{H}(\sigma_{min})$ given by $f \mapsto f.\sigma_{alg}$ is an isomorphism according to Lemma 1.4 \cite{MR2290601} and  by Corollary 7.2 \cite{Pyv1} we have a canonical isomorphism $\mathfrak{Z}_{\Omega} \simeq \mathcal{H}(\sigma_{min}(\lambda))$. Composing $\beta$ with those isomorphisms gives us the desired map.

\end{enumerate}

Notice that it follows from step 5 that the map $\beta$ does not depend on monodromy. 

The organisation of this section is the following. In the subsection \ref{LA.2.1} we will construct the map $\beta$ in full generality by modifying the argument given in \cite{MR3529394}. Then in section \ref{LA.2.2} we will give another proof of the same result in the Iwahori case based on explicit computation.

\subsection{Construction in general case}\label{LA.2.1}

The aim of the next sections is to carry out steps 1.-8. outlined above. Morally proof is based on the Lemma 4.3 \cite{MR3529394}. We will prove the following Theorem:

\begin{thm}\label{4.4}
There is an $E$-algebra homomorphism
\[\beta: \mathcal{H}(\sigma_{min}) \longrightarrow R_{\tilde{\mathfrak{p}}}^{\square}(\sigma_{min})[1/p]\]

\noindent such that for any closed point $x$ of $R_{\tilde{\mathfrak{p}}}^{\square}(\sigma_{min})[1/p]$ with residue field $E_x$, the action of $\mathcal{H}(\sigma_{min})$ on a smooth $G$-representation $\pi_{sm}(r_x)$ factors as $\beta$ composed with the evaluation map  $R_{\tilde{\mathfrak{p}}}^{\square}(\sigma_{min})[1/p] \longrightarrow E_x$.
\end{thm}

\begin{proof} Consider the following map, obtained by specialisation:
\[ \gamma_G : \mathfrak{Z}_{\Omega} \longrightarrow \prod_{x \in \Spm R^{\square}_{\overline{r}}(\tau, \mathbf{v})[1/p]} E'_x\]
where $\gamma_G$ is defined on the factor corresponding to $x$ by evaluating $\mathfrak{Z}_{\Omega}$ at the closed point in the Bernstein component $\Omega$ determined via local Langlands by $x$, and $E'_x/E_x$ is a sufficiently large finite extension.

\noindent Consider as well the following map, also obtained by specialisation:
\[ \gamma_{WD} : R^{\square}_{\overline{r}}(\tau, \mathbf{v})[1/p] \longrightarrow \prod_{x \in \Spm R^{\square}_{\overline{r}}(\tau, \mathbf{v})[1/p]} E'_x\]

\noindent The map $\gamma_{WD}$ is injective, because the ring $R^{\square}_{\overline{r}}(\tau, \mathbf{v})[1/p]$ is reduced and Jacobson.

\noindent We have the following diagram:
$$\xymatrixrowsep{5pc}\xymatrixcolsep{5pc}\xymatrix{
\mathfrak{Z}_{\Omega}  \ar[r]^-{\gamma_G} \ar@{.>}[rd]^{I} &\prod\limits_{x \in \Spm R^{\square}_{\overline{r}}(\tau, \mathbf{v})[1/p]} E'_x\\
W_F \ar[u]^T \ar@{.>}[r]^{?} &R^{\square}_{\overline{r}}(\tau, \mathbf{v})[1/p] \ar[u]^{\gamma_{WD}}}
$$

\noindent where $T : W_F  \longrightarrow  \mathfrak{Z}_{\Omega}$ be the pseudo-representation constructed in Proposition 3.11 of \cite{Che} and $I$ is the map that we want to construct. Observe that \cite[Lemma 3.24]{MR3529394} tells us that the Chenevier’s $E[\mathfrak{B}]$ is our  $\mathfrak{Z}_{\Omega}$, so that the definition of the map $T$ makes sense.

First we will construct a map $?$ such that the diagram above commutes.  We can apply the Fontaine's recipe to the absolute Frobenius $\varphi$ on $D^{\square}_{\overline{r}}(\tau, \mathbf{v})$, which is a free $R^{\square}_{\overline{r}}(\tau, \mathbf{v})[1/p]\otimes_{\Qp}F_{0}$-module of rank $n$. Let's recall first this construction in the usual setting.

Let $x \in \Spm R^{\square}_{\overline{r}}(\tau, \mathbf{v})[1/p]$ be a closed $E$-valued point. Let $D_x$ and $\varphi_x$ be the specializations of $D^{\square}_{\overline{r}}(\tau, \mathbf{v})$ and $\varphi$ at $x$, respectively. Let $L$ be a finite extension of $F$ where all the Galois representations of the given inertial type $\tau$ are semi-stable and $L_0$ a subfield of $L$ such that $L/L_0$ is totally ramified.

Then we deduce from the isomorphism 
$$L_0\otimes_{\Qp} E \simeq \prod_{\sigma_0 : L_0 \hookrightarrow E} E,$$

\noindent the isomorphism
\[D_x =  \prod_{\sigma_0 : L_0 \hookrightarrow E} D_{\sigma_0},\]

\noindent where $D_{\sigma_0} = (0,\ldots,0,1_{\sigma_0},0,\ldots,0 )D_x$, is the \textquotedblleft $\sigma_{0}$-th coordinate of $D_x$". Fix now a $\sigma_0$. Set $W_x = D_{\sigma_0}$.

Let $w \in W_F$, define $\overline{w}$ to be the image of $w$ in $\mathrm{Gal}(L/F)$ and let $\alpha(w) \in f_0\Z$ be such that the action of $w$ on $\overline{\mathbb{F}}_p$ is the $\alpha(w)$-power of the map $(x\mapsto x^{p})$.

We can define an endomorphism of $D_x$ by $r_x(w):=\overline{w}\circ \varphi_x^{-\alpha(w)}$. It can be shown that the restriction of $r(w)$ to $W_x$ does not depend on $\sigma_0$.

We are interested in the trace of $r_x(w)$. We have trivially $$\mathrm{Tr}(r_x(w)|D_x) = |\mathrm{Hom}_{\Qp}(L_0,E)| \mathrm{Tr}(r_x(w)|W_x).$$ \noindent  However since $E$ is assumed to be large enough we have $|\mathrm{Hom}_{\Qp}(L_0,E)| = [L_0:\Qp]$.

Observe that it makes sense to define for each $w \in W_F$ an endomorphism $r(w): =\overline{w}\circ \varphi^{-\alpha(w)}$ of $D^{\square}_{\overline{r}}(\tau, \mathbf{v})$ and we can also take its trace.

Define now the following map:
$$\begin{array}{ccccc}
\mathrm{Tr} & : & W_F & \longrightarrow & R^{\square}_{\overline{r}}(\tau, \mathbf{v})[1/p] \\
 & & w &  \longmapsto & \frac{1}{[L_0:\Qp]}\mathrm{Tr}(r(w)|D^{\square}_{\overline{r}}(\tau, \mathbf{v})) \\
\end{array}$$

\noindent Then by the construction of $T$, we have $\gamma_G \circ T = \gamma_{WD} \circ \mathrm{Tr}$, i.e. the diagram of sets
$$\xymatrixrowsep{5pc}\xymatrixcolsep{5pc}\xymatrix{
\mathfrak{Z}_{\Omega}  \ar[r]^-{\gamma_G} \ar@{.>}[rd]^{I} &\prod\limits_{x \in \Spm R^{\square}_{\overline{r}}(\tau, \mathbf{v})[1/p]} E'_x\\
W_F \ar[u]^T \ar[r]^-{\mathrm{Tr}} &R^{\square}_{\overline{r}}(\tau, \mathbf{v})[1/p] \ar[u]^{\gamma_{WD}}}
$$

\noindent commutes. Now we can define the map $I$. In order to do so, it suffices to show that the image of  $\mathfrak{Z}_{\Omega}$ under $\gamma_G$ is contained in the image of $\gamma_{WD}$. However by Lemma 4.5 \cite{MR3529394} the image of $T$ generates $\mathfrak{Z}_{\Omega}$. Then, any element $a\in \mathfrak{Z}_{\Omega}$ can be written as $a = \sum_i \mu_i T(g_i)$. It follows from the commutative diagram above, that the map $I$ is given by $\sum_i \mu_i T(g_i) \mapsto \sum_i \mu_i \mathrm{Tr}(g_i)$. Let's prove that the map $I$ is a well defined ring homomorphism. 

The map $I$ is well defined. Indeed, if we choose two different presentations of an element $a\in \mathfrak{Z}_{\Omega}$, $a=\sum_i \mu_i T(g_i)=\sum_k \lambda_k T(h_k)$ then the elements $\sum_i \mu_i \mathrm{Tr}(g_i)$ and $\sum_k \lambda_k \mathrm{Tr}(h_k)$ should coincide. It is enough to prove that if $\sum_i \mu_i T(g_i)=0$, then $\sum_i \mu_i \mathrm{Tr}(g_i)=0$. Indeed, we have $0=\gamma_G(0)=\gamma_G(\sum_i \mu_i T(g_i))=\sum_i \mu_i \gamma_G(T(g_i))= \sum_i \mu_i \gamma_{WD}(\mathrm{Tr}(g_i))=\gamma_{WD}(\sum_i \mu_i \mathrm{Tr}(g_i))$, then $\sum_i \mu_i \mathrm{Tr}(g_i)=0$ since $\gamma_{WD}$ is injective.

Now, we will prove that $I$ is a ring homomorphism. First notice that $ \gamma_G$ and $\gamma_{WD}$ are already ring homomorphisms. Let any $a, b \in \mathfrak{Z}_{\Omega}$, then 
$$\gamma_{WD}(I(a.b)-I(a).I(b))=\gamma_{WD}(I(a.b))-\gamma_{WD}(I(a)).\gamma_{WD}(I(b))$$
$$=\gamma_G(a.b)-\gamma_G(a).\gamma_G(b)=0$$
\noindent Since $\gamma_{WD}$ is injective it follows that $I(a.b)=I(a).I(b)$. Similarly we get $I(a+b)=I(a)+I(b)$ and $I(1)=1$.

Let $M$ be the Levi subgroup in the supercuspidal support of any irreducible representation in $\Omega$, and $X (M)$ be the group of unramified characters of $M$. The group automorphism $X (M) \longrightarrow X (M)$ given by $\chi_M \rightarrow \chi_M |\det|^{\frac{(1-n)}{2}}$ gives rise to an $E$-isomorphism $\Sp \mathfrak{Z}_D \rightarrow \Sp \mathfrak{Z}_D$. The latter map is invariant under the $W(D)$-action (the point is that $|\det|$ is invariant under $G$-conjugation) so it descends to an $E$-isomorphism $\Sp \mathfrak{Z}_{\Omega} \rightarrow \Sp \mathfrak{Z}_{\Omega}$. Let $t_W : \mathfrak{Z}_{\Omega} \rightarrow \mathfrak{Z}_{\Omega}$ denote the induced isomorphism. Now we construct $\beta'$ as the following composite map:
\[ \mathfrak{Z}_{\Omega}\xrightarrow{t_W} \mathfrak{Z}_{\Omega} \xrightarrow{I} R^{\square}_{\overline{r}}(\tau, \mathbf{v})[1/p]\]

In order to get the map $\beta$ as in the statement of the theorem, compose $\beta'$ with the isomorphisms  $\mathcal{H}(\sigma_{min}(\lambda)) \to \mathcal{H}(\sigma_{min})$ and $\mathfrak{Z}_{\Omega} \simeq \mathcal{H}(\sigma_{min}(\lambda))$. The the desired interpolation property of $\beta'$, follows from the commutative diagram above. This can be easily be checked on points.
\end{proof}

\subsection{Construction in the Iwahori case}\label{LA.2.2}

In this subsection we give an explicit construction of the map $\beta: \mathcal{H}(\sigma_{min}) \longrightarrow R_{\tilde{\mathfrak{p}}}^{\square}(\sigma_{min})[1/p]$ in the semi-stable case.

Assume now, that $\pi$ has a trivial type $(I,1)$, i.e. $\pi^{I}\neq  0$ and $\Omega = [T,1]_G$. So the inertial type $\tau$ is also trivial.  Let $\mathcal{H}(\sigma_{min}):=\mathrm{End}_{G}(\mathrm{c\text{--} Ind}_{K}^{G} \sigma_{min})$, and by Corollary 7.2 \cite{Pyv1} we have a canonical isomorphism $\mathfrak{Z}_{\Omega} \simeq \mathcal{H}(\sigma_{min}(\lambda))$ and also $\mathfrak{Z}_{\Omega} \simeq \mathcal{H}(\sigma_{max}(\lambda))=\mathcal{H}(G,K)$.  Moreover the map $\mathcal{H}(G,K) \to \mathcal{H}(\sigma_{max})$ given by  $f \mapsto f.\sigma_{alg}$. is an isomorphism according to Lemma 1.4 \cite{MR2290601}. By Satake isomorphism we have $\mathcal{H}(G,K) \simeq E[\theta_1,\ldots,\theta_{n-1},(\theta_{n})^{\pm 1}]$, where $\theta_r$ is a double coset operator $ \left[K \begin{pmatrix}
 \varpi I_{r}& 0 \\ 
0 & I_{n-r}
\end{pmatrix} K \right] $. Putting all these isomorphisms together we have $\mathcal{H}(\sigma_{min}(\lambda)) \simeq \mathfrak{Z}_{\Omega}\simeq E[\theta_1,\ldots,\theta_{n-1},(\theta_{n})^{\pm 1}]$. So in order to describe completely the action of $\mathfrak{Z}_{\Omega}$ on $\eta \otimes |\det|^{\frac{n-1}{2}}$, it would be enough to describe the action of each $\theta_r$.

Let $q$ be the cardinality of residue field $\mathcal{O}_F/\mathfrak{p}_F$ where $\mathcal{O}_F$ is the ring of integers of $F$ and $\mathfrak{p}_F$ the maximal ideal. Let $\varpi$ be a uniformizer of $F$.

\medskip
We describe first the action of $\mathfrak{Z}_{\Omega}$.

\begin{lemma}
Let $\psi := \psi_1 \otimes \ldots \otimes \psi_n$, an unramified character of torus $T$, and $\eta = i_{B}^{G}(\psi)$. Then $\theta_r$ acts on $\eta \otimes |\det|^{\frac{n-1}{2}}$ by a scalar:
\[q^{\frac{r(1-r)}{2}}\sum\limits_{\lambda_1 < \ldots < \lambda_r} \psi_{\lambda_1}(\varpi)\ldots\psi_{\lambda_r}(\varpi) \]

\noindent where the sum is taken through all the integers $1 \leq \lambda_i \leq n$ such that those inequalities are satisfied.
\end{lemma}

\begin{proof} We follow closely Bump's lecture notes \cite{Bump} on Hecke algebras, and adapts the argument therein for our needs. One may consult section 9, Proposition 40 in \cite{Bump} for more details. It follows from Iwasawa decomposition that the space of $K$-invariants of $(\eta \otimes |\det|^{\frac{n-1}{2}})^K$ is one dimensional and generated by the function $f^{\circ} : bk \mapsto \delta_B^{1/2}(b)\psi'(b)$, with $b \in B$ and $k \in K$ and $\psi'(b)=\psi_1(b_{11})|b_{11}|^{\frac{n-1}{2}}\ldots\psi_n(b_{nn})|b_{nn}|^{\frac{n-1}{2}}$. Hence $\theta_r.f^{\circ}=c.f^{\circ}$, then $c = \theta_r.f^{\circ}(1)$. Using the a double coset decomposition:
\[K \begin{pmatrix}
 \varpi I_{r}& 0 \\ 
0 & I_{n-r}
\end{pmatrix} K =  \bigcup_{\beta \in \Lambda} \beta K,\]

\noindent where $\Lambda$ is a complete set of representatives, we will compute $\theta_r.f^{\circ}(1)$. We have a freedom of choice for $\beta$'s, so we can put them in a specific form. More precisely we have 
\[K \begin{pmatrix}
 \varpi I_{r}& 0 \\ 
0 & I_{n-r}
\end{pmatrix} K = \bigcup_{S= \{\lambda_1,\ldots,\lambda_r\}} \bigcup_{\beta \in \Lambda_S} \beta K,\]

\noindent where $\lambda_1 < \ldots < \lambda_r$ and $\beta \in \Lambda_S$ if and only if the following four conditions are satisfied:
\begin{enumerate}
\item $\beta$ is upper triangular;
\item $\beta_{ii}=\varpi$ if $i\in S$ and $\beta_{ii}=1$ if $i \notin S$;
\item $\beta_{ij}$ is a representative of an element of $\mathcal{O}_F/\mathfrak{p}_F$ if $i<j$, $i\in S$ and $j \notin S$
\item all other entries are zero
\end{enumerate}

The number of non-zero entries outside diagonal in a matrix $\beta \in \Lambda_S$ is $\sum\limits_{i=1}^{r}(n-r-\lambda_i+i)=r(n-r)+r(r+1)/2-\sum\limits_{i=1}^{r}\lambda_i$. Therefore from
\[ |\Lambda_S| = q^{r(n-r)+r(r+1)/2-\sum\limits_{i=1}^{r}\lambda_i}\]

\noindent it follows that
\[\theta_r.f^{\circ}(1) = \sum\limits_{\lambda_1 < \ldots < \lambda_r} \sum\limits_{\beta \in \Lambda_S} f^{\circ}(\beta) =  \sum\limits_{\lambda_1 < \ldots < \lambda_r} |\Lambda_S| f^{\circ}(\beta). \]

\noindent Now let's compute $f^{\circ}(\beta) = \delta_B^{1/2}(\beta) \psi'(\beta)$. By definition we have 
\[ \delta_B^{1/2}(\beta)= \prod\limits_{i=1}^{n} |\beta_{ii}|^{\frac{n-2i+1}{2}} = q^{-\sum\limits_{i=1}^{r}\frac{n-2\lambda_i+1}{2}}=q^{-\frac{r(n+1)}{2}+\sum\limits_{i=1}^{r}\lambda_i},\]

\noindent and 
\[\psi'(\beta) = q^{-\frac{r(n-1)}{2}}\psi_{\lambda_1}(\varpi)\ldots\psi_{\lambda_r}(\varpi)\]

\noindent The total power of $q$ is:
\[ r(n-r)+r(r+1)/2-\sum\limits_{i=1}^{r}\lambda_i -\frac{r(n+1)}{2}+\sum\limits_{i=1}^{r}\lambda_i -\frac{r(n-1)}{2} = - \frac{r(r-1)}{2}\]

\noindent Finally
\[\theta_r.f^{\circ}(1) = q^{\frac{r(1-r)}{2}}\sum\limits_{\lambda_1 < \ldots < \lambda_r} \psi_{\lambda_1}(\varpi)\ldots\psi_{\lambda_r}(\varpi) \]
\end{proof}

Let  $x :R^{\square}_{\overline{r}}(\tau, \mathbf{v})[1/p] \longrightarrow E$ be an $E$-algebra homomorphism with $V_x$ the corresponding $n$-dimensional Galois representation of $G_F$. Here $V_x$ is already semi-stable, so $L=F$ and $f=f_0$. Let $D_{st}(V_x):=(B_{st} \otimes_{\Qp} V_x)^{G_F}$, by construction this is also $D_x$, the specialization of $D^{\square}_{\overline{r}}(\tau, \mathbf{v})$ at the closed point $x$. Then by Proposition \ref{2.1}, $WD(D_{st}(V_x))$ is the Weil-Deligne representation that corresponds to $\pi$ by the local Langlands correspondence, with normalization as in \cite{MR1876802}. Assume that $\pi$ is a generic representation. Let $\mathfrak{\eta}:= \mathrm{c\text{--} Ind}_{K}^{G} \sigma_{max} \otimes_{\mathfrak{Z}_{\Omega},\chi_{\pi}}E$ as in Lemma \ref{4.2}. The admissible filtered $(\varphi, N, \mathrm{Gal}(L/F))$-module $D_x=D_{st}(V_x)$ is equipped with Frobenius endomorphism $\phi_x$, which is the specialization of the universal Frobenius $\varphi$ on  $D^{\square}_{\overline{r}}(\tau, \mathbf{v})$ at $x$, i.e. $\varphi \otimes \kappa(x)=\phi_x$.

\begin{prop}\label{4.3}
Let  $x :R^{\square}_{\overline{r}}(\tau, \mathbf{v})[1/p] \longrightarrow E$, be an $E$-algebra homomorphism as above. The double coset operator $\theta_{r}$ acts on $\eta\otimes |\det|^{\frac{n-1}{2}}$ (equivalently on $\pi \otimes |\det|^{\frac{n-1}{2}}$) as scalar multiplication by $q^{\frac{r(1-r)}{2}}\mathrm{Tr}(\bigwedge^{r} (\phi_x)^f)$.
\end{prop}

\begin{proof} With the notations of Lemma \ref{4.2} we have $s=1$ and $\pi_1=1$. Then there is a partition of $n$, $\sum_{i=1}^{t}n_i=n$, such that $\pi: = Q(\Delta_1) \times \ldots \times Q(\Delta_t)$, with $\Delta_i=\chi_i\otimes\ldots\otimes\chi_i|\cdot|^{n_i-1}$ and  $\chi_i\chi_j^{-1}\neq |\cdot|^{\pm 1}$ for all $i\neq j$. Then $\eta= \tilde{\Delta}_1 \times \ldots \times  \tilde{\Delta}_t$, where $\tilde{\Delta}_i=\chi_i|\cdot|^{1-n_i}\otimes\ldots\otimes\chi_i$. Define $\psi := \psi_1 \otimes \ldots \otimes \psi_n = \tilde{\Delta}_1 \otimes \ldots\otimes \tilde{\Delta}_t$ an unramified character of torus $T$, so that $\eta \simeq i_{B}^{G}(\psi)$.

By previous lemma, $\theta_{r}$ acts on one dimensional space $(\eta\otimes |\det|^{\frac{n-1}{2}})^{K}$ as scalar multiplication by  
\[q^{\frac{r(1-r)}{2}} s_r(\chi_1(\varpi)q^{n_{1}-1},\ldots,\chi_1(\varpi),\ldots,\chi_t(\varpi)q^{n_t-1},\ldots,\chi_t(\varpi))\]

\noindent where $s_r $ is the $r^{\mathrm{th}}$ symmetric polynomial in $n$ variables.

The eigenvalues of $WD(D_{st}(V_x))(Frob_{p})$ are $\chi_1(\varpi)q^{n_{1}-1}$,\ldots,$\chi_1(\varpi)$,\\\ldots,$\chi_t(\varpi)q^{n_t-1}$,\ldots,$\chi_t(\varpi)$, via the classical local Langlands correspondence. Then it follows that
$$s_r(\chi_1(\varpi)q^{n_{1}-1},\ldots,\chi_1(\varpi),\ldots,\chi_t(\varpi)q^{n_t-1},\ldots,\chi_t(\varpi))$$
$$=\mathrm{Tr}(\bigwedge^{r} WD(D_{st}(V_x))(Frob_{p}))=\mathrm{Tr}(\bigwedge^{r} (\phi_x^f))$$

\noindent where $Frob_p$ is the geometric Frobenius. Notice that the computations above do not depend on the choice of $Frob_p$.
\end{proof}

If for an embedding $\sigma$ the Hodge-Tate weights are $i_{\kappa, 1}<\ldots<i_{\kappa ,n}$, define $\xi_{j,\kappa}=-i_{\kappa, j} + (j-1)$. The highest weight of the algebraic representation $\sigma_{alg}$ with respect to the upper triangular matrices is given by $\mathrm{diag}(x_1,\ldots,x_n) \mapsto \prod_{j=1}^{n}\prod_{\kappa} \kappa(x_{j}^{\xi_{\kappa, j}})$. Then we have to rescale $\theta_{r}$ by the factor
$$\varpi^{-\sum_{\kappa}\sum\limits_{j=r}^{n}\xi_{\kappa,j}}$$
\noindent in order to be compatible with isomorphism, $\mathcal{H}(\sigma_{min}(\lambda)) \to \mathcal{H}(\sigma_{min})$ given by $f \mapsto f.\sigma_{alg}$.

Define $\tilde{\theta}_r = q^{\frac{r(r-1)}{2}}.\varpi^{-\sum_{\kappa}\sum\limits_{i=r}^{n}\xi_{\kappa, j}}.\theta_r$. Then we have a canonical  isomorphism $\mathcal{H}(\sigma_{min}) \simeq E[\tilde{\theta}_1,\ldots,\tilde{\theta}_{n-1},(\tilde{\theta}_{n})^{\pm 1}]$. We can summarize the results of this section with the following theorem:

\begin{thm}
If $\tau $ is trivial, then define  $\beta:\mathcal{H}(\sigma_{min}) \longrightarrow R^{\square}_{\overline{r}}(\tau, \mathbf{v})[1/p]$ by the assignment 
$$\tilde{\theta}_r \mapsto \varpi^{-\sum_{\kappa}\sum\limits_{i=r}^{n}\xi_{\kappa, j}}\mathrm{Tr}(\bigwedge^{r} \varphi^f),$$

\noindent where $\varphi$ is the universal Frobenius on $D^{\square}_{\overline{r}}(\tau, \mathbf{v})$. Then the map $\beta$ is an $E$-algebra homomorphism and $\beta$ interpolates the local Langlands correspondence, i.e. for any closed point $x$ of $R_{\tilde{\mathfrak{p}}}^{\square}(\sigma_{min})[1/p]$ with residue field $E_x$, the action of $\mathcal{H}(\sigma_{min})$ on a smooth $G$-representation $\pi_{sm}(r_x)$ factors as $\beta$ composed with the evaluation map  $R_{\tilde{\mathfrak{p}}}^{\square}(\sigma_{min})[1/p] \longrightarrow E_x$.
\end{thm}

\begin{proof}
Since $E[\tilde{\theta}_1,\ldots,\tilde{\theta}_{n-1},(\tilde{\theta}_{n})^{\pm 1}]$ is a polynomial $E$-algebra, the previous assignment $\beta$ is a ring homomorphism. Moreover the weak admissibility of $D_x$ implies that $val_F(\varpi^{-\sum_{\sigma}\sum\limits_{i=r}^{n}\xi_{j,\sigma}}\mathrm{Tr}(\bigwedge^{r} \phi_x^f)) \geq 0$, and so $x(\beta(\theta_r))$ belongs to the ring of integers, for all $r$. It follows that the image of the map $\beta$ is contained in the normalization of $R^{\square}_{\overline{r}}(\tau, \mathbf{v})[1/p]$, by Proposition 7.3.6 \cite{MR1383213}.

As it was observed in point 7. in section \ref{LA.2.1} such a map interpolates the local Langlands correspondence on all the closed points.
\end{proof}

\section{Local deformation rings}\label{LA.3}

We begin this section with some elementary linear algebra. Those preparatory results will help us to deal with monodromy of potentially semi-stable Galois representations. Indeed we will introduce locally algebraic representations $\sigma_{\mathcal{P}}$, where $\mathcal{P}$ is a partition-valued function. The properties of smooth part $\sigma_{\mathcal{P}}(\lambda)$ of $\sigma_{\mathcal{P}}$ were described in \cite[Theorem 3.7]{MR3769675} and it was explained how the monodromy of an irreducible generic representation can be read of the $\sigma_{\mathcal{P}}(\lambda)$'s that it contains. In a similar way, we may study the support of $M_{\infty}(\sigma_{min}^{\circ})$ by introducing a stratification that depends on the $\sigma_{\mathcal{P}}$'s. This will be dealt with in the next section. Here we will introduce a stratification of $R_{\tilde{\mathfrak{p}}}^{\square}(\sigma_{min})$ with respect to any partition-valued function $\mathcal{P}$, more precisely we will construct the rings  $R^{\square}_{\tilde{\mathfrak{p}}}(\sigma_{\mathcal{P}})$, which are reduced, $p$-torsion free quotient of $R^{\square}_{\overline{r}}(\tau, \mathbf{v})$, satisfying the following property: $x \in \Sp R^{\square}_{\tilde{\mathfrak{p}}}(\sigma_{\mathcal{P}})[1/p]$ if and only if $\mathcal{P}_x \geq \mathcal{P}$. 

\medskip

Recall a few facts about partitions. Let $(\lambda_1,\ldots,\lambda_l)$ be a partition of $n$, i.e. we have $n=\lambda_1+\ldots+\lambda_l$ with $\lambda_1\geq \ldots \geq \lambda_l>0$. We say that a partition $\lambda^{c}$ is conjugate of $\lambda=(\lambda_{1},\ldots, \lambda_{l})$ if it is represented by the reflected diagram of the one associated to $\lambda$ with respect to the line $y=-x$ with the coordinate of the upper left corner is taken to be $(0,0)$. We have that $\lambda_i^{c} = |\left\lbrace k : \lambda_k \geq i\right\rbrace |$.

Let $M$ be any field, $V$ a $n$-dimensional $M$-vector space and $N:V \to V$ a nilpotent endomorphism. Then the Jordan normal form of $N$ is uniquely determined up conjugacy by a partition $(n_1,\ldots,n_t)$, i.e the blocks are ordered by decreasing size $n_1 \geq \ldots \geq n_t$.

\begin{lemma}\label{4.5}
Let $M$ be any field, $V$ a $n$-dimensional $M$-vector space, with two nilpotent endomorphisms $N:V \to V$ and $N':V \to V$. Let $(n_1,\ldots,n_t)$ (resp. $(n'_1,\ldots,n'_s)$) be a partition that encodes the Jordan normal form of $N$ (resp.$N'$). Then the following statements are equivalent:
\begin{enumerate}
\item $\forall i$, $\dim \mathrm{Ker}(N^{i}) \leq \dim \mathrm{Ker}(N'^{i})$.
\item $\forall i$, $\sum\limits_{k=1}^{i} n_k \geq \sum\limits_{k=1}^{i} n'_k$.
\end{enumerate}
\end{lemma}
\begin{proof} The Jordan normal form gives an isomorphism $N \simeq \bigoplus_{k=1}^{t}N_k$, where $N_k$ is a nilpotent operator of maximal rank on a $n_k$-dimensional vector space. Then:
\[\dim \mathrm{Ker}(N_k^{i})= \begin{cases}
        i, & \text{for } i\leq n_k\\
        n_k, & \text{for } i> n_k
        \end{cases}\]

\noindent and $\dim \mathrm{Ker}(N^{i})= \sum\limits_{k=1}^{t} \dim \mathrm{Ker}(N_k^{i}) =\sum\limits_{k=1}^{t} \min(i,n_k)$. 

\noindent Let $\kappa_j = \dim \mathrm{Ker}(N^{j}) - \dim \mathrm{Ker}(N^{j-1})$ for $j \geq 1$ and $\dim \mathrm{Ker}(N^{0})=0$. We get $(\kappa_1,\kappa_2,\ldots)$ a partition of $n$ and we will call this partition a kernel partition of the nilpotent operator $N$. By the description of $\dim \mathrm{Ker}(N^{i})$ in terms of the partition $(n_1,\ldots,n_t)$, we see that $\kappa_i = |\left\lbrace k : n_k \geq i\right\rbrace |$. So the partition $(\kappa_1,\kappa_2,\ldots)$ is the dual of the partition $(n_1,\ldots,n_t)$. Let now $(\kappa_1',\kappa_2',\ldots)$ be the kernel partition of $N'$. Then the inequalities:
$$\dim \mathrm{Ker}(N^{i})=\sum\limits_{j=1}^{i} \kappa_j \leq \dim \mathrm{Ker}(N'^{i})= \sum\limits_{j=1}^{i} \kappa_j',$$

\noindent $\forall i$, are equivalent to the inequalities from 2. This concludes the proof.
\end{proof}

\begin{lemma}\label{4.6}
Let $A$ be a commutative ring, $V$ projective finitely generated $A$-module and $N: V \longrightarrow V$ a nilpotent $A$-linear operator. Then the set 
\[\left\lbrace \mathfrak{p} \in \Sp A | \dim_{\kappa(\mathfrak{p})}(\mathrm{Coker} N) \otimes_{A} \kappa(\mathfrak{p}) \geq m \right\rbrace \]
\noindent is closed for any integer $m$. 

For a point $x \in \Sp A$, the shape (Jordan normal form) of nilpotent operator $N \otimes \kappa(x)$ is given by a partition $\mathcal{P}_x$ and this partition determines uniquely, up to conjugacy, a Jordan normal form of a nilpotent operator. Define a partial order $\leq$ on partitions which is the reverse of so-called natural or dominance partial order (\cite{MR3077154} chapter 5 section 5.1.4). Then for all integers $i$,
\[\dim_{\kappa(x)}(\mathrm{Coker} N^{i}) \otimes_{A} \kappa(x) \leq \dim_{\kappa(y)}(\mathrm{Coker} N^{i}) \otimes_{A} \kappa(y)\]

\noindent if and only if $\mathcal{P}_x \leq \mathcal{P}_y$.

\end{lemma}
\begin{proof} Let's prove the first assertion. Let $m_1,\ldots m_n$, any set of generators of  $C:=\mathrm{Coker} N$ over $A$. It would be enough to prove that the set $U:= \left\lbrace \mathfrak{p} \in \Sp A | \dim_{\kappa(\mathfrak{p})}C \otimes_{A} \kappa(\mathfrak{p}) < n \right\rbrace$ is open. Let $\mathfrak{p} \in \Sp A$ and $\bar{x}_1,\ldots,\bar{x}_k $ be a basis of  $\kappa(\mathfrak{p})$-vector space $C_{\mathfrak{p}}/\mathfrak{p}C_{\mathfrak{p}}$. It follows from Nakayama's lemma the lifts $x_1,\ldots,x_k$ to $C_{\mathfrak{p}}$, form a minimal generating set of $C_{\mathfrak{p}}$ over $A_{\mathfrak{p}}$. Write $m_i/1 = \sum\limits_{j=1}^{k} (a_{ij}/b_{ij})x_j$ and let $b=\prod b_{ij}$. For any $\mathfrak{q} \in D(b)$, $x_1,\ldots,x_k$ is still a generating set of $C_{\mathfrak{q}}$ over $A_{\mathfrak{q}}$. Again, by Nakayama's lemma it follows that $\dim_{\kappa(\mathfrak{p})}C \otimes_{A} \kappa(\mathfrak{p}) \leq k < n$, so that $D(b) \subseteq U$. Therefore $U$ is open. 

The second assertion follows from the previous lemma, because \\ \noindent $\dim \mathrm{Ker}(N^{i} \otimes \kappa(x))= \dim \mathrm{Coker}(N^{i} \otimes \kappa(x))$ and we have an isomorphism $\mathrm{Coker}(N^{i} \otimes \kappa(x)) \simeq (\mathrm{Coker} N^{i}) \otimes_{A} \kappa(x)$ since the tensor product is right-exact.

\end{proof}

Recall from the previous section that we have an endomorphism $\varphi$ on $D^{\square}_{\overline{r}}(\tau, \mathbf{v})$. Again by Theorem (2.5.5)\cite{MR2373358}, there is a universal monodromy operator $N : D^{\square}_{\overline{r}}(\tau, \mathbf{v}) \longrightarrow D^{\square}_{\overline{r}}(\tau, \mathbf{v})$ which is $F_{0}\otimes_{\Qp}R^{\square}_{\overline{r}}(\tau, \mathbf{v})[1/p]$-linear. Observe, that the monodromy of $WD(r_x)$ is the specialization of $N$ at a closed point $x \in \Sp R^{\square}_{\overline{r}}(\tau, \mathbf{v})[1/p]$.

Let $\mathcal{P}$ be a partition-valued function as in \cite{MR1728541}. Apply the previous lemma with $A=F_{0}\otimes_{\Qp}R^{\square}_{\overline{r}}(\tau, \mathbf{v})[1/p]$ and $V=D^{\square}_{\overline{r}}(\tau, \mathbf{v})$ to get that the set 
\[\left\lbrace x \in \Sp R^{\square}_{\overline{r}}(\tau, \mathbf{v})[1/p] | \mathcal{P}_x \geq \mathcal{P} \right\rbrace\]
 \[= \bigcap_{i \geq 1} \left\lbrace x \in \Sp R^{\square}_{\overline{r}}(\tau, \mathbf{v})[1/p] | \dim_{\kappa(x)}(\mathrm{Coker} N^{i}) \otimes_{A} \kappa(x) \geq m_i \right\rbrace \]

\noindent is closed, with $m_i=\sum_{\sigma}\sum_{k}\min(i,\mathcal{P}(\sigma)(k))$. Hence this set, is of the form $V(I_{\mathcal{P}})$,  where $I_{\mathcal{P}}$ is an ideal in $R^{\square}_{\overline{r}}(\tau, \mathbf{v})$, such that the quotient is reduced. We can now make the following definition:

\begin{definition}
Define the ring $R^{\square}_{\tilde{\mathfrak{p}}}(\sigma_{\mathcal{P}}):=R^{\square}_{\overline{r}}(\tau, \mathbf{v})/I_{\mathcal{P}}$, this ring has the following property: $x \in \Sp R^{\square}_{\tilde{\mathfrak{p}}}(\sigma_{\mathcal{P}})[1/p]$ if and only if $\mathcal{P}_x \geq \mathcal{P}$. 
\end{definition}

Observe that $R^{\square}_{\tilde{\mathfrak{p}}}(\sigma_{\mathcal{P}})$ is a reduced, $p$-torsion free quotient of $R^{\square}_{\overline{r}}(\tau, \mathbf{v})$. For $\mathcal{P}$ maximal partition, which we denote by $\sigma_{\mathcal{P}}=\sigma_{max}$, we get a potentially crystalline deformation ring and for $\mathcal{P}$ minimal, which we denote by $\sigma_{\mathcal{P}}=\sigma_{min}$, we get a potentially semi-stable deformation ring $R_{\tilde{\mathfrak{p}}}^{\square}(\sigma_{min}):=R^{\square}_{\overline{r}}(\tau, \mathbf{v})$. 

An easy consequence of the construction above is the following lemma:

\begin{lemma}
We have  $\dim R_{\tilde{\mathfrak{p}}}^{\square}(\sigma_{min}) = \dim R_{\tilde{\mathfrak{p}}}^{\square}(\sigma_{\mathcal{P}})$.
\end{lemma}
\begin{proof} Notice that Theorem (3.3.4) \cite{MR2373358} gives the dimension of $R_{\tilde{\mathfrak{p}}}^{\square}(\sigma_{min})$ and by Theorem (3.3.8) \cite{MR2373358} we know that  $\dim R_{\tilde{\mathfrak{p}}}^{\square}(\sigma_{max}) = \dim R_{\tilde{\mathfrak{p}}}^{\square}(\sigma_{min})$. Since we have, $R_{\tilde{\mathfrak{p}}}^{\square}(\sigma_{min})  \twoheadrightarrow R_{\tilde{\mathfrak{p}}}^{\square}(\sigma_{\mathcal{P}}) \twoheadrightarrow R_{\tilde{\mathfrak{p}}}^{\square}(\sigma_{max})$. 

It follows that $\dim R_{\tilde{\mathfrak{p}}}^{\square}(\sigma_{max})\leq \dim R_{\tilde{\mathfrak{p}}}^{\square}(\sigma_{\mathcal{P}}) \leq \dim R_{\tilde{\mathfrak{p}}}^{\square}(\sigma_{min})$, so that $\dim R_{\tilde{\mathfrak{p}}}^{\square}(\sigma_{min}) = \dim R_{\tilde{\mathfrak{p}}}^{\square}(\sigma_{\mathcal{P}})$.
\end{proof}

\section{Local-Global compatibility}\label{LA.4}

We will begin this section by recalling a few important definitions from \cite{MR3529394}. In \cite[section 2.1]{MR3529394} the authors construct a suitable globalization $\overline{\rho} :  G_{\tilde{F}^{+}} \longrightarrow \mathcal{G}_n(\mathbb{F})$ of  $\overline{r} : G_F \longrightarrow GL_n(\mathbb{F})$,  where $\tilde{F}$ is an imaginary CM field with maximal totally real subfield $\tilde{F}^{+}$, and $\mathcal{G}_n$ is a group scheme over $\Z$ defined as a semi-direct product of $GL_n \times GL_1$ by the group$\{1,j \}$ (cf. section 1.8, p.14 of \textit{op. cit.}). Then theauthors use $\overline{\rho}$ to carry out the Taylor--Wiles--Kisin patching method to construct $M_{\infty}$ (cf. section 2.8, p.29 of \textit{op. cit.}).
For the purpose of this paper it is enough to view $M_{\infty}$ as an $R_{\infty}[G]$-module, such that it is finitely generated over $R_{\infty}[[GL_n(\Z_p)]]$, where $R_\infty$(cf. \cite[section 2.8, p.27]{MR3529394}) is a complete Noetherian local $R_{\tilde{\mathfrak{p}}}^{\square}$-algebra with residue field $\mathbb{F}$. Finally let us observe that we need to assume that $\overline{r}$ has a potentially diagonalizable lift in order to construct $\overline{\rho}$. However this assumption is not needed for our main result in the light of \cite[Remark 2.15]{MR3529394}.

In this section we will study the support of $M_{\infty}(\sigma_{min}^{\circ})$ by introducing a stratification that depends on the $\sigma_{\mathcal{P}}$'s. This will allow us to have finer control on the monodromy operator. 

The main result of this section is  Theorem \ref{4.10}. This result tells us that the action of $\mathcal{H}(\sigma_{min})$ on $M_{\infty}(\sigma_{min}^{\circ})$ is compatible with the interpolation map $\mathcal{H}(\sigma_{min}) \longrightarrow R_{\tilde{\mathfrak{p}}}^{\square}(\sigma_{min})[1/p]$, constructed previously. Most of the proofs in this section are very similar to the ones given in \cite[section 4]{MR3529394}. 

\medskip

Let $\mathcal{P}$ be a partition-valued function. Define $\sigma_{\mathcal{P}} := \sigma_{\mathcal{P}}(\lambda) \otimes \sigma_{alg}$, so that $(\sigma_{\mathcal{P}})_{sm}=\sigma_{\mathcal{P}}(\lambda)$ and $(\sigma_{\mathcal{P}})_{alg}=\sigma_{alg}$, where $\sigma_{\mathcal{P}}(\lambda)$ is a smooth type for $K$ as in \cite{MR1728541} and $\sigma_{alg}$ is the restriction to $K$ of an irreducible algebraic representation of $\mathrm{Res}_{F/\Qp}GL_n$.  Fix a $K$-stable $\mathcal{O}$-lattice $\sigma_{\mathcal{P}}^{\circ}$ in $\sigma_{\mathcal{P}}$. Set
\[M_{\infty}(\sigma_{\mathcal{P}}^{\circ}):= \left( \mathrm{Hom}_{\mathcal{O}[[K]]}^{cont}(M_{\infty}, (\sigma_{\mathcal{P}}^{\circ})^{d})\right) ^{d}\]

\noindent where we are considering homomorphisms that are continuous for the profinite topology on $M_{\infty}$ and the $p$-adic topology on $(\sigma_{\mathcal{P}}^{\circ})^{d}$, and where we equip $\mathrm{Hom}_{\mathcal{O}[[K]]}^{cont}(M_{\infty}, (\sigma_{\mathcal{P}}^{\circ})^{d})$ with the $p$-adic topology. Note that $M_{\infty}(\sigma_{\mathcal{P}}^{\circ})$ is an $\mathcal{O}$-torsion free, profinite, linear-topological $\mathcal{O}$-module.

\medskip
Let $R_{\infty}(\sigma_{\mathcal{P}})$ be the quotient of $R_\infty$ which acts faithfully on $M_{\infty}(\sigma_{\mathcal{P}}^{\circ})$, i.e. $R_{\infty}(\sigma_{\mathcal{P}})=M_{\infty}(\sigma_{\mathcal{P}}^{\circ})/\mathrm{ann}(M_{\infty}(\sigma_{\mathcal{P}}^{\circ}))$. Set $R_{\infty}(\sigma_{\mathcal{P}})'=R_{\infty}\otimes_{R_{\tilde{\mathfrak{p}}}^{\square}}R_{\tilde{\mathfrak{p}}}^{\square}(\sigma_{\mathcal{P}})$.

\begin{lemma}\label{4.7}

Let $\mathcal{P}$ be a partition-valued function, then $R_{\infty}(\sigma_{\mathcal{P}})$ is a reduced $\mathcal{O}$-torsion free quotient of $R_{\infty}(\sigma_{\mathcal{P}})'$. Moreover the module $M_{\infty}(\sigma_{\mathcal{P}}^{\circ})$ is Cohen-Macaulay.

\end{lemma}
\begin{proof} That $R_{\infty}(\sigma_{\mathcal{P}})$ is $\mathcal{O}$-torsion free follows immediately from the fact that by definition it acts faithfully on the $\mathcal{O}$-torsion free module $M_{\infty}(\sigma_{\mathcal{P}}^{\circ})$.

The fact that it is actually a quotient of $R_{\infty}(\sigma_{\mathcal{P}})'$ is a consequence of classical local-global compatibility at $\tilde{\mathfrak{p}}$. The proof of this is identical to the proof of Lemma 4.17(1) in \cite{MR3529394}. Even though that proof is written for $\sigma$ (i.e. $\sigma_{max}$ with the notation of this paper), all the details remain unchanged if we replace $\sigma$ by $\sigma_{\mathcal{P}}$, if we observe that by local-global compatibility (Theorem 1.1 of \cite{MR3272276}) the restriction to the local factor at $\tilde{\mathfrak{p}}$ of the global Galois representation coming from a closed point of a Hecke algebra,  is potentially semi-stable such that the partition-valued function associated to monodromy (as in Lemma \ref{4.6}) of this local Galois representation bigger then $\mathcal{P}$.

To prove the remaining assertions first notice that the module $M_{\infty}(\lambda^{\circ})$ is a Cohen-Macaulay module,  by Lemma 4.30 \cite{MR3529394}. Then $M_{\infty}(\sigma_{\mathcal{P}}^{\circ})$ is also a Cohen-Macaulay module because it is a direct summand of $M_{\infty}(\lambda^{\circ})$. Finally, to see that $R_{\infty}(\sigma_{\mathcal{P}})$ is reduced, notice that since $R_{\infty}(\sigma_{\mathcal{P}})'$ is reduced, any non-reduced quotient of the same dimension will have an associated prime, which is not minimal. Since $M_{\infty}(\sigma_{\mathcal{P}}^{\circ})$ is a faithful Cohen-Macaulay module over $R_{\infty}(\sigma_{\mathcal{P}})$, this cannot happen, and so $R_{\infty}(\sigma_{\mathcal{P}})$ is reduced.
\end{proof}

Let $\mathcal{H}(\sigma_{min}^{\circ}) := \mathrm{End}_{G}(\mathrm{c\text{--} Ind}_{K}^{G} \sigma_{min}^{\circ})$. Note that since $\sigma_{\mathcal{P}}^{\circ}$ is a free $\mathcal{O}$-module of finite rank, it follows from the proof of Theorem 1.2 of \cite{MR1900706} that Schikhof duality induces an isomorphism
\[\mathrm{Hom}_{\mathcal{O}[[K]]}^{cont}(M_{\infty}, (\sigma_{\mathcal{P}}^{\circ})^{d}) \simeq \mathrm{Hom}_{K}(\sigma_{\mathcal{P}}^{\circ},(M_{\infty})^{d})\]

\noindent and Frobenius reciprocity gives 
$$\mathrm{Hom}_{K}(\sigma_{\mathcal{P}}^{\circ},(M_{\infty})^{d})= \mathrm{Hom}_{G}(\mathrm{c\text{--} Ind}_{K}^{G} \sigma_{\mathcal{P}}^{\circ},(M_{\infty})^{d}).$$

\noindent Thus $M_{\infty}(\sigma_{\mathcal{P}}^{\circ})$ is equipped with an action of $\mathfrak{Z}_{\Omega}$ which commutes with the action of $R_{\infty}$.

When $\sigma_{\mathcal{P}}= \sigma_{min}$, the module $M_{\infty}(\sigma_{min}^{\circ})$ is equipped with an action of  $\mathcal{H}(\sigma_{min}^{\circ})$. Such an action of $\mathcal{H}(\sigma_{min}^{\circ})$ commutes with the action of $R_{\infty}$. The isomorphism $\mathcal{H}(\sigma_{\min}) \simeq \mathfrak{Z}_{\Omega}$(Corollary 7.2 \cite{Pyv1}) and the isomorphism of Lemma 1.4 \cite{MR2290601}, $\mathcal{H}(\sigma_{min}(\lambda)) \to \mathcal{H}(\sigma_{min})$, allow us to define the action of $\mathcal{H}(\sigma_{min}^{\circ})$ on $M_{\infty}(\sigma_{\mathcal{P}}^{\circ})$.

\begin{lemma}\label{4.8}
If $z \in \mathcal{H}(\sigma_{min}^{\circ})$ is such that $\beta(z) \in R_{\tilde{\mathfrak{p}}}^{\square}(\sigma_{\mathcal{P}})$, then the action of $z$ on $M_{\infty}(\sigma_{\mathcal{P}}^{\circ})$ agrees with the action of $\beta(z)$ via the natural map $R_{\tilde{\mathfrak{p}}}^{\square}(\sigma_{\mathcal{P}}) \longrightarrow R_{\infty}(\sigma_{\mathcal{P}})'$
\end{lemma}

\begin{proof} As before, this is a consequence of classical local-global compatibility at $\tilde{\mathfrak{p}}$. The proof of this is identical to the proof of Lemma 4.17(2) in \cite{MR3529394}, where we replace $\sigma$ by $\sigma_{\mathcal{P}}$ and instead of using Lemma 4.17(1) in \cite{MR3529394} we apply Lemma \ref{4.7}. 
\end{proof}

We will now define the space of algebraic automorphic forms. First recall some notation from \cite{MR3529394}. The globalization constructed in section 2.1 \cite{MR3529394}, provides us with a global imaginary CM field $\tilde{F}$ with maximal totally real subfield $\tilde{F}^{+}$. We refer the reader to this section for the details of these definitions.

Recall some notation from section 2.3 \cite{MR3529394}. Let $\tilde{G} / \tilde{F}^{+}$ a certain definite unitary group as defined in the paper \cite{MR3529394}. Let $U= \prod\limits_{v} U_v$ be a compact open subgroup of $\tilde{G}(\A_{\tilde{F}^{+}}^{\infty})$. Let $S_p$ denote the set of primes of $\tilde{F}^{+}$ dividing $p$. Fix $\mathfrak{p}|p$. Let $\xi$ the weight as in section \ref{I.1} of this article and $\tau$ the inertial type as in section \ref{I.3}. Let $W_{\xi,\tau}$ be the finite free $\mathcal{O}$-module with an action of $\prod\limits_{v \in S_p \setminus \{\mathfrak{p}\}} U_v$.

For any compact open $U$ and any $\mathcal{O}$-module $V$, let $S_{\xi,\tau}(U, V)$ denote the set of continuous functions
$$f: \tilde{G}(\tilde{F}^{+})\setminus \tilde{G}(\A_{\tilde{F}^{+}}^{\infty}) \longrightarrow W_{\xi,\tau} \otimes V$$

\noindent such that for $g \in \tilde{G}(\A_{\tilde{F}^{+}}^{\infty})$ we have $f(gu) = u^{-1}f(g)$ for $u \in U$, where $U$ acts on $W_{\xi,\tau} \otimes V$ via the projection to $\prod\limits_{v \in S_p} U_v$. The space $S_{\xi,\tau}(U, V)$ is called a \textbf{space of algebraic modular forms}.

Now we can define $\varpi$-adically completed cohomology space. For each positive integer $m$, the compact open subgroups $U_m$ as defined in the beginning of \cite[section 2.3]{MR3529394} have the same level away from $\mathfrak{p}$. Let $U^{\mathfrak{p}}$ denote that common level. Define the ̟$\varpi$-adically \textbf{completed cohomology space}:
$$\tilde{S}_{\xi,\tau}(U^{\mathfrak{p}},\mathcal{O})_{\mathfrak{m}}:=\varprojlim_{s}(\varinjlim_{m} S_{\xi,\tau}(U_m, \mathcal{O}/\varpi^{s})_{\mathfrak{m}}).$$

The space is equipped with a natural $G$-action, induced from the action of $G$ on algebraic automorphic forms.

The module $M_{\infty}$ comes with an action of $S_{\infty}$ (cf.  \cite[page 27, Section 2.8]{MR3529394}). Recall that by  \cite[Corollary 2.11]{MR3529394}, we have a $G$-equivariant isomorphism $M_{\infty}/\mathfrak{a}M_{\infty} \simeq \tilde{S}_{\xi,\tau}(U^{\mathfrak{p}},\mathcal{O})_{\mathfrak{m}}^{d}$, where $\mathfrak{a}$ is the augmentation ideal in $S_{\infty}$ generated by some formal variables (cf. \cite[page 27, Section 2.8]{MR3529394}). Moreover that isomorphism commutes with the $R_{\tilde{\mathfrak{p}}}^{\square}$-action on both sides.

\begin{lemma}\label{4.11}
Let $pr: \Sp R_{\infty}(\sigma_{min})' [1/p] \longrightarrow \Sp R_{\tilde{\mathfrak{p}}}^{\square}(\sigma_{min})[1/p]$ be the map induced by $j:R_{\tilde{\mathfrak{p}}}^{\square}(\sigma_{min}) \longrightarrow R_{\infty}\otimes_{R_{\tilde{\mathfrak{p}}}^{\square}}R_{\tilde{\mathfrak{p}}}^{\square}(\sigma_{min})$, $a\mapsto 1 \otimes a$. Let $x \in \Sp R_{\tilde{\mathfrak{p}}}^{\square}(\sigma_{min})[1/p]$ be a closed smooth point. Then any $y \in pr^{-1}(x) \subset  R_{\infty}(\sigma_{min})'[1/p]$ is a smooth point of $\Sp R_{\infty}(\sigma_{min})$.
\end{lemma}

\begin{proof} This is essentially the first part of the proof of  \cite[Theorem 4.35]{MR3529394}.
\end{proof}

\begin{prop}\label{4.26}
If $y \in \Sp R_{\infty}(\sigma_{min})'[1/p] \cap V(\mathfrak{a})$ is  a closed point, then $y$ is a smooth point of $\Sp R_{\infty}(\sigma_{min})$ and $V(r_{y})^{l.alg} \simeq \pi_{sm}(r_y) \otimes \pi_{alg}(r_y)$.
\end{prop}

\begin{proof} We follow here quite closely the proof of Theorem 4.35 \cite{MR3529394}. By definition
$$V(r_{y}):=\mathrm{Hom}_{\mathcal{O}}^{cont}(M_{\infty} \otimes_{R_{\infty},y}\mathcal{O},E)$$

\noindent Since $\mathfrak{a} \subseteq \mathrm{Ker}(y)=\mathfrak{m}_y$, we have that:
$$\mathrm{Hom}_{\mathcal{O}}^{cont}(M_{\infty} \otimes_{R_{\infty},y}\mathcal{O},E)=\mathrm{Hom}_{\mathcal{O}}^{cont}(M_{\infty}/\mathfrak{a}M_{\infty} \otimes_{R_{\infty},y}\mathcal{O},E) $$

\noindent Then by Corollary 2.11\cite{MR3529394}, we have
$$\mathrm{Hom}_{\mathcal{O}}^{cont}(M_{\infty}/\mathfrak{a}M_{\infty} \otimes_{R_{\infty},y}\mathcal{O},E) \simeq \mathrm{Hom}_{\mathcal{O}}^{cont}(\tilde{S}_{\xi,\tau}(U^{\mathfrak{p}},\mathcal{O})_{\mathfrak{m}}^{d} \otimes_{R_{\infty},y}\mathcal{O},E)$$

The ideal $\mathfrak{m}_y$ is finitely generated, choose a presentation $\mathfrak{m}_y=(a_1,\ldots,a_k)$, then we get an exact sequence of $R_{\infty}$-modules:
$$R_{\infty}^{\oplus k}\longrightarrow R_{\infty} \longrightarrow \mathcal{O} \longrightarrow 0$$

\noindent Let $\Pi(\bullet) = \mathrm{Hom}_{\mathcal{O}}^{cont}(\bullet \otimes_{R_{\infty}} \tilde{S}_{\xi,\tau}(U^{\mathfrak{p}},\mathcal{O})_{\mathfrak{m}}^{d},E)$. Then the functor $\Pi$ is left exact and contravariant, by Lemma 2.20 of \cite{MR3306557}. Apply this functor to the exact sequence above to get the following exact sequence:
$$ 0 \longrightarrow V(r_y) \longrightarrow  \mathrm{Hom}_{\mathcal{O}}^{cont}( \tilde{S}_{\xi,\tau}(U^{\mathfrak{p}},\mathcal{O})_{\mathfrak{m}}^{d},E) \xrightarrow{f} \mathrm{Hom}_{\mathcal{O}}^{cont}(\tilde{S}_{\xi,\tau}(U^{\mathfrak{p}},\mathcal{O})_{\mathfrak{m}}^{d},E)^{\oplus k}$$

\noindent were $f(l) = (l.a_1,\ldots,l.a_k)$. By the exactness we identify 
$$V(r_y) \simeq \mathrm{Hom}_{\mathcal{O}}^{cont}( \tilde{S}_{\xi,\tau}(U^{\mathfrak{p}},\mathcal{O})_{\mathfrak{m}}^{d},E)[\mathfrak{m}_y],$$ 

\noindent but  
$$\tilde{S}_{\xi,\tau}(U^{\mathfrak{p}},\mathcal{O})_{\mathfrak{m}} \otimes_{\mathcal O} E  \simeq \mathrm{Hom}_{\mathcal{O}}^{cont}( \tilde{S}_{\xi,\tau}(U^{\mathfrak{p}},\mathcal{O})_{\mathfrak{m}}^{d},\mathcal{O})\otimes_{\mathcal O} E$$ 
$$ \simeq \mathrm{Hom}_{\mathcal{O}}^{cont}( \tilde{S}_{\xi,\tau}(U^{\mathfrak{p}},\mathcal{O})_{\mathfrak{m}}^{d},E)$$

\noindent Thus
\[V(r_{y})^{l.alg} \simeq (\tilde{S}_{\xi,\tau}(U^{\mathfrak{p}},\mathcal{O})\otimes_{\mathcal{O}}E)^{l.alg}[\mathfrak{m}_y]\]

Proposition 3.2.4 of \cite{MR2207783} shows that locally algebraic vectors of any given weight are precisely the algebraic automorphic forms of that weight. Hence:
\[ (\tilde{S}_{\xi,\tau}(U^{\mathfrak{p}},\mathcal{O})\otimes_{\mathcal{O}}E)^{l.alg}[\mathfrak{m}_y] \simeq \pi_{sm}(r_y) \otimes \pi_{alg}(r_y) \]

This isomorphism follows from classical local-global compatibility (Theorem 1.1 of \cite{MR3272276}). A priori, $\pi_{sm}(r_y) \otimes \pi_{alg}(r_y)$ may appear with some multiplicity. However this multiplicity is seen to be one. Indeed the group $\tilde{G}$ is compact at infinity, so the condition ($\ast$) from Theorem 5.4 of \cite{MR2856380} is automatically satisfied. We may then apply Theorems 5.4 and 5.9 of \cite{MR2856380}, where $\sigma$, in those Theorems, is our $(\tilde{S}_{\xi,\tau}(U^{\mathfrak{p}},\mathcal{O})\otimes_{\mathcal{O}}E)^{l.alg}$ and $\pi$ is an automorphic cuspidal representation of $GL_n$, which is a base change of $\sigma$. Then  by the choice of $U^{\mathfrak{p}}$ (section 2.3 \cite{MR3529394} for definition of the $U_m$), the fact that we have fixed the action mod $p$ of the Hecke operators at $\tilde{v}_1$(section 2.3 \cite{MR3529394}), and the irreducibility of the globalization of $\overline{r}$, we see that the multiplicity of $\pi_{sm}(r_y) \otimes \pi_{alg}(r_y)$ is one.

Local factors of $\pi$, as in the paragraph above, are generic according to Corollary of Theorem 5.5 \cite{MR0348047}. Then by Theorem 5.9 \cite{MR2856380}, the local factors of $(\tilde{S}_{\xi,\tau}(U^{\mathfrak{p}},\mathcal{O})\otimes_{\mathcal{O}}E)^{l.alg}$ are also generic, since $\tilde{G}$, by construction, is quasi-split at all the finite places. It follows that $\pi_{sm}(r_y)$ is generic. Moreover, by Theorem 1.2.7\cite{MR3546966}, the closed point $pr(y)$ ($pr: \Sp R_{\infty}(\sigma_{min})' [1/p] \longrightarrow \Sp R_{\tilde{\mathfrak{p}}}^{\square}(\sigma_{min})[1/p]$, with notation of Lemma \ref{4.11}) is smooth if and only if $\pi_{sm}(r_y)$ is generic. Then by Lemma \ref{4.11} the point $y$ is smooth.
\end{proof}

In order to study the support of $M_{\infty}(\sigma_{min}^{\circ})$, we will use the commutative algebra arguments underlying the Taylor-Wiles-Kisin method.

\begin{prop}\label{4.9}
\begin{enumerate}
\item The module $M_{\infty}(\sigma_{min}^{\circ})[1/p]$ is locally free of rank one over the regular locus of $R_{\infty}(\sigma_{min})[1/p]$.
\item $\Sp R_{\infty}(\sigma_{min}) [1/p]$ is a union of irreducible components of 

$\Sp R_{\infty}(\sigma_{min})'[1/p]$.
\end{enumerate}
\end{prop}
\begin{proof} \textbf{(1)}. Let $ {\mathfrak{m}}$ be a smooth point in the support of $M_{\infty}(\sigma_{min}^{\circ})[1/p]$. Since $M_{\infty}(\sigma_{min}^{\circ})[1/p]$ is a Cohen-Macaulay module we have $\mathrm{depth} \: M_{\infty}(\sigma_{min}^{\circ})[1/p]_{\mathfrak{m}} = \dim M_{\infty}(\sigma_{min}^{\circ})[1/p]_{\mathfrak{m}}$. Moreover $\dim M_{\infty}(\sigma_{min}^{\circ})[1/p]_{\mathfrak{m}} = \dim R_{\infty}(\sigma_{min})[1/p]_{\mathfrak{m}}$ since $R_{\infty}(\sigma_{min})$ acts faithfully on  $M_{\infty}(\sigma_{min}^{\circ})$. 

By assumption the ring $R_{\infty}(\sigma_{min})[1/p]_{\mathfrak{m}}$ is regular, it follows that the module $M_{\infty}(\sigma_{min}^{\circ})[1/p]_{\mathfrak{m}}$ has finite projective dimension over this ring. We also have that 
$\mathrm{depth} \: R_{\infty}(\sigma_{min})[1/p]_{\mathfrak{m}} = \dim R_{\infty}(\sigma_{min})[1/p]_{\mathfrak{m}}$. Then by \\Auslander-Buchsbaum formula (Theorem 19.1 \cite{MR1011461}), $M_{\infty}(\sigma_{min}^{\circ})[1/p]_{\mathfrak{m}}$ is free over $R_{\infty}(\sigma_{min})[1/p]_{\mathfrak{m}}$. It follows that $M_{\infty}(\sigma_{min}^{\circ})[1/p]$ is locally free (i.e. projective) over the regular locus of $R_{\infty}(\sigma_{min})[1/p]$.

Let's check that it is locally free of rank one. Let $x\in \mathrm{Supp}\: M_{\infty}(\sigma_{min}^{\circ})$ and $y \in \mathrm{Supp}( M_{\infty}(\sigma_{min}^{\circ})) \cap V(\mathfrak{a})$ a smooth closed point that lies on the same irreducible component $V$ as $x$. Such a point $y$ always exists because $V\cap V(\mathfrak{a})\neq 0$. Indeed by Proposition \ref{4.26}, if $y \in V\cap V(\mathfrak{a})$, then $y$ is smooth. Since $M_{\infty}(\sigma_{min}^{\circ})[1/p]$ is projective the local rank is constant on irreducible components of the support. It would be enough to compute the local rank at $y$, which is given by
\[\dim_{E} M_{\infty}(\sigma_{min}^{\circ}) \otimes_{R_{\infty}} \kappa(y) = \dim_E \mathrm{Hom}_K(\sigma_{min},V(r_y)^{l.alg}),\]

\noindent according to Proposition 2.22 \cite{MR3306557}. By Proposition \ref{4.26} we have that $V(r_{y})^{l.alg} \simeq \pi_{sm}(r_y) \otimes \pi_{alg}(r_y)$. Moreover, since $\sigma_{alg}$ is an irreducible representation of a Lie algebra of $G$, we have 
\[\dim_E \mathrm{Hom}_K(\sigma_{min},V(r_y)^{l.alg})=\dim_E \mathrm{Hom}_K(\sigma_{min}(\lambda),\pi_{sm}(r_y))\]

Then $\dim_E \mathrm{Hom}_K(\sigma_{min}(\lambda),\pi_{sm}(r_y))=1$ by Lemma 3.2 \cite{Pyv2}, because $\pi_{sm}(r_y)$ is generic.

\textbf{(2)}. The proof is the same as in Lemma 4.18 \cite{MR3529394}.
\end{proof}

\noindent The action of $\mathfrak{Z}_{\Omega}$ on $M_{\infty}(\sigma_{min}^{\circ})[1/p]$ induces an $E$-algebra homomorphism:
$$\alpha : \mathfrak{Z}_{\Omega} \longrightarrow \mathrm{End}_{R_{\infty}[1/p]}(M_{\infty}(\sigma_{min}^{\circ})[1/p])$$
\noindent From the Proposition \ref{4.9}, we deduce that:

\begin{thm}\label{4.10}
We have the following commutative diagram:
$$\xymatrix{
(\Sp R_{\infty}(\sigma_{min}) [1/p])^{reg} \ar@{^{(}->}[d] \ar[r]^-{\alpha^{\sharp}} & \Sp \mathcal{H}(\sigma_{min}) \\
\Sp R_{\infty}(\sigma_{min}) [1/p] \ar[r]^{pr} &\Sp R_{\tilde{\mathfrak{p}}}^{\square}(\sigma_{min})[1/p], \ar[u] }$$

\noindent where $(\Sp R_{\infty}(\sigma_{min}) [1/p])^{reg}$ is the regular locus of $\Sp R_{\infty}(\sigma_{min}) [1/p]$, $\alpha^{\sharp}$ the map induced by $\alpha$ and $pr: \Sp R_{\infty}(\sigma_{min})' [1/p] \longrightarrow \Sp R_{\tilde{\mathfrak{p}}}^{\square}(\sigma_{min})[1/p]$ the  map induced by $j:R_{\tilde{\mathfrak{p}}}^{\square}(\sigma_{min}) \longrightarrow R_{\infty}\otimes_{R_{\tilde{\mathfrak{p}}}^{\square}}R_{\tilde{\mathfrak{p}}}^{\square}(\sigma_{min})$, $x\mapsto 1 \otimes x$.
\end{thm}

\begin{proof} We proceed here as in the proof of Theorem 4.19 \cite{MR3529394}. It is enough to check all it on points since all the rings are Jacobson and reduced. Let $x: R_{\infty}(\sigma_{min}) [1/p] \rightarrow E$ a closed point smooth point. Note firstly that if $z \in \mathcal{H}(\sigma_{min}^{\circ})$ is such that $\beta(z) \in R_{\tilde{\mathfrak{p}}}^{\square}(\sigma_{min})$, then $x(\alpha(z))=x(j(\beta(z)))$ by Lemma \ref{4.8}. Since $R_\infty(\sigma_{min})$ is $p$-torsion free by Lemma \ref{4.7}, it is therefore enough to show that $\mathcal{H}(\sigma_{min}^{\circ})$ is spanned over $E$ by such elements. But, $\mathcal{H}(\sigma_{min}^{\circ})$ certainly spans $\mathcal{H}(\sigma_{min})$ over $E$, so it is enough to show that for any element $z \in \mathcal{H}(\sigma_{min}^{\circ})$ , we have $\beta(p^{C}z) \in R_{\tilde{\mathfrak{p}}}^{\square}(\sigma_{min})$ for some $C \geq 0$. The latter condition is obviously true, this concludes the proof.
\end{proof}

\section{Support of patched modules}\label{LA.5}

Let $(J,\lambda)$ be the type, a locally algebraic representation $\lambda \otimes (\sigma_{alg}|J)$ of $J$ will be again denoted by $\lambda$. We have also a patched module  $M_{\infty}(\lambda^{\circ}):=\left( \mathrm{Hom}_{\mathcal{O}[[J]]}^{cont}(M_{\infty}, (\lambda^{\circ})^{d})\right) ^{d}$, where $\lambda^{\circ}$ is a $J$-stable lattice in $\lambda$. Define also $R_{\infty}(\lambda):=R_{\infty}/\mathrm{ann}(M_{\infty}(\lambda^{\circ}))$. We would like to have some statements about a support of patched modules. More precisely, we will prove that $M_{\infty}(\sigma_{min}^{\circ})$ and $M_{\infty}(\lambda^{\circ})$ have the same support.

\begin{prop}\label{4.33}
$$\mathrm{Supp}(M_{\infty}(\sigma_{min}^{\circ})) = \mathrm{Supp}(M_{\infty}(\lambda^{\circ}))$$
\end{prop}

\begin{proof} It follows from decomposition:
\[\mathrm{Ind}_J^K \lambda = \bigoplus_{\mathcal{P}} \sigma_{\mathcal{P}}^{\oplus m_{\mathcal{P}}}\]

\noindent that 
\[M_{\infty}(\lambda^{\circ})= \bigoplus_{\mathcal{P}} M_{\infty}(\sigma_{\mathcal{P}}^{\circ})^{\oplus m_{\mathcal{P}}}.\]

\noindent Then $\mathrm{Supp}(M_{\infty}(\sigma_{min}^{\circ})) \subseteq \mathrm{Supp}(M_{\infty}(\lambda^{\circ}))$. By definition, $\mathrm{Supp}(M_{\infty}(\sigma_{min}^{\circ})) = \Sp R_{\infty}(\sigma_{min})$ and also $\mathrm{Supp}(M_{\infty}(\lambda^{\circ})) = \Sp R_{\infty}(\lambda)$. Let $V(\mathfrak{p})$ be an irreducible component of the spectrum  $\Sp R_{\infty}(\lambda)$. It is enough to find a point $x \in V(\mathfrak{p})$ such that $x \notin V(\mathfrak{q})$ for any minimal prime $\mathfrak{q}$ of $R_{\infty}(\lambda)$ such that $\mathfrak{q} \neq \mathfrak{p}$ and $x \in \Sp R_{\infty}(\sigma_{min})$. The ideal $\mathfrak{a}$ is generated by a regular sequence $(y_1,\ldots,y_h)$ and $y_1$,\ldots, $y_h$, $\varpi$ is a system of parameters for $\Sp R_{\infty}(\lambda)/\mathfrak{p}$. Then by Lemma 3.9 \cite{MR3544298}, $V(\mathfrak{p})$ contains a closed point $x \in  \Sp R_{\infty}(\lambda)/(y_1,\ldots y_h)[1/p]$. The point $x$ is smooth by the Lemma \ref{4.26}, hence it does not lie on the intersection of irreducible components.

We have that $x \in  \Sp R_{\infty}(\lambda)[1/p] \cap V(\mathfrak{a})$, so it is a closed point of $\mathrm{Supp}(M_{\infty}(\lambda^{\circ}))$, then 
$$M_{\infty}(\lambda^{\circ}) \otimes_{R_{\infty}} \kappa(x)\neq 0$$

\noindent and by Proposition 2.22 \cite{MR3306557}, we have that:
$$M_{\infty}(\lambda^{\circ}) \otimes_{R_{\infty}} \kappa(x) = \mathrm{Hom}_{E}^{cont}\left(\mathrm{Hom}_{J} (\lambda, V(r_{x})^{l.alg}),E\right) \neq 0$$

\noindent Then by Proposition \ref{4.26} we have that $V(r_{x})^{l.alg} \simeq \pi_{sm}(r_x) \otimes \pi_{alg}(r_x)$. Moreover the representation $\pi_{sm}(r_x)$ is generic,  it follows then from Theorem 3.1 \cite{Pyv2} that we also have $\mathrm{Hom}_{K}(\sigma_{min}, V(r_{x})^{l.alg}) \neq 0$. This means that $x \in  \Sp R_{\infty}(\sigma_{min})[1/p] \cap V(\mathfrak{a})$. 
\end{proof}

\section{Computation of locally algebraic vectors}\label{LA.6}

By Proposition \ref{4.33}, we have $\mathrm{Supp}(M_{\infty}(\sigma_{min}^{\circ})) = \mathrm{Supp}(M_{\infty}(\lambda^{\circ}))$. In what follows we always identify these two sets, so we have $\Sp R_{\infty}(\sigma_{min}) = \mathrm{Supp}(M_{\infty}(\sigma_{min}^{\circ})) = \mathrm{Supp}(M_{\infty}(\lambda^{\circ})) = \Sp R_{\infty}(\lambda)$. Let $x\in \Spm R_{\infty} [1/p]$ be such that $V(r_{x}) \neq 0$. Assume moreover that $x \in \mathrm{Supp}(M_{\infty}(\lambda^{\circ}))$ and that the representation  $\pi_{sm}(r_{x}):=r_{p}^{-1} (WD(r_x))$ is generic and irreducible. By definition $\pi_{sm}(r_x)$ lies in $\Omega $.

As always for any partition-valued function $\mathcal{P}$, we will write $\sigma_{\mathcal{P}} := \sigma_{\mathcal{P}}(\lambda) \otimes \sigma_{alg}$, so that $(\sigma_{\mathcal{P}})_{sm}=\sigma_{\mathcal{P}}(\lambda)$ and $(\sigma_{\mathcal{P}})_{alg}=\sigma_{alg}$.

By Proposition 4.33 \cite{MR3529394}, we have $V(r_{x})^{l.alg} = \pi_{x} \otimes \pi_{alg}(r_x)$, where $\pi_{x}$ is an admissible smooth representation which lies in $\Omega$.

Since $x \in \mathrm{Supp}(M_{\infty}(\sigma_{min}^{\circ}))$, then by Proposition 2.22 \cite{MR3306557} we have that $0 \neq \mathrm{Hom}_{K}(\sigma_{min}, V(r_{x})^{l.alg}) = \mathrm{Hom}_{K}(\sigma_{min}(\lambda),\pi_{x})$

The action of $\mathfrak{Z}_{\Omega}$ on $\pi_{sm}(r_x)$ defines a $E$-algebra morphism $ \chi_{sm} : \mathfrak{Z}_{\Omega} \longrightarrow \mathrm{End}_{G}(\pi_{sm}(r_x))\simeq E$, the kernel of such a morphism is a maximal ideal in $\mathfrak{Z}_{\Omega}$.


The space $M_{\infty}(\sigma_{min}^{\circ}) \otimes_{R_{\infty}} \kappa(x)$ is one dimensional by Proposition \ref{4.9} and the Hecke algebra $\mathcal{H}(\sigma_{min}(\lambda))$ acts on this space by a character $\chi' : \mathcal{H}(\sigma_{min}(\lambda)) \longrightarrow E$. Composing $\chi'$ with an isomorphism $\mathfrak{Z}_{\Omega} \simeq \mathcal{H}(\sigma_{min})$, obtained from Lemma 1.4 \cite{MR2290601} and Corollary 7.2 \cite{Pyv1}, we get an $E$-algebra morphism $ \chi : \mathfrak{Z}_{\Omega} \longrightarrow  E$.

\begin{lemma}\label{4.24}
The $E$-algebra morphisms $\chi$ and $\chi_{sm}$ defined above coincide.
\end{lemma}

\begin{proof} It follows from Theorem \ref{4.10} and Theorem \ref{4.4} and the isomorphism $\mathfrak{Z}_{\Omega} \simeq \mathcal{H}(\sigma_{min})$ as above. 
\end{proof}

\begin{lemma}\label{4.25}
The representation $\pi_{sm}(r_x)$ is a $G$-subquotient of $\pi_{x}$.
\end{lemma}

\begin{proof} Define $\gamma_{x}:= \mathrm{c\text{--} Ind}_{K}^{G} (\sigma_{min}(\lambda)) \otimes_{\mathfrak{Z}_{\Omega},\chi} E$. Since $x \in \mathrm{Supp}(M_{\infty}(\sigma_{min}^{\circ}))$, we have by definition $0 \neq \mathrm{Hom}_{K}(\sigma_{min},V(r_{x})^{l.alg}) = \mathrm{Hom}_{K}(\sigma_{min}(\lambda),\pi_{x})$. So, there exists a non zero map $\psi: \gamma_{x} \longrightarrow \pi_{x}$.

Let $\pi'$ be any irreducible quotient of $\gamma_x$; then $\mathrm{Hom}_{K}(\sigma_{min}(\lambda),\pi') \neq 0$, by Theorem 3.1 \cite{Pyv2}  since $\pi'$ is generic. It follows that by Corollary 3.11 \cite{MR3529394}, that the representation $\pi'$ is the socle of $\mathrm{c\text{--} Ind}_{K}^{G} \sigma_{max}(\lambda) \otimes_{\mathfrak{Z}_{\Omega},\chi} E$. We write it $\pi' \simeq \mathrm{soc}_{G}(\mathrm{c\text{--} Ind}_{K}^{G} (\sigma_{max}(\lambda)) \otimes_{\mathfrak{Z}_{\Omega},\chi} E)$. Similarly by corollary 3.11 \cite{MR3529394}, $\pi_{sm}(r_x) \simeq \mathrm{soc}_{G}(\mathrm{c\text{--} Ind}_{K}^{G} (\sigma_{max}(\lambda)) \otimes_{\mathfrak{Z}_{\Omega},\chi_{sm}} E)$. By Lemma \ref{4.24} $\chi_{sm} = \chi$, then $\pi' \simeq \pi_{sm}(r_x)$. So at this stage we proved that the cosocle of $\gamma_x$ is generic, irreducible and isomorphic to $\pi_{sm}(r_x)$.

Let $\kappa = \mathrm{Ker}(\psi)$, then $\gamma_x/\kappa\hookrightarrow \pi_{x}$. Let now $\pi'$ be any irreducible quotient of $\gamma_x/\kappa$; in particular $\pi'$ is a sub-quotient of $\pi_{x}$. Moreover $\gamma_x \twoheadrightarrow \gamma_x/\kappa \twoheadrightarrow \pi'$, so $\pi'$ is an irreducible quotient of $\gamma_x$. By what  we have proven above, $\pi'\simeq \pi_{sm}(r_x)$. It follows that $\pi_{sm}(r_x)$ is a sub-quotient of $\pi_{x}$.
\end{proof}

\begin{prop}\label{4.27}
Let $x$, $y$ be two closed, $E$-valued points of $\Sp R_{\infty} (\sigma_{min})[1/p]$, lying on the same irreducible component. Let $\mathcal{P}$ be a partition-valued function. If $x$ is smooth, then
$$\dim_E \mathrm{Hom}_{K}(\sigma_{\mathcal{P}}, V(r_x)^{l.alg}) \leq \dim_E \mathrm{Hom}_{K}(\sigma_{\mathcal{P}}, V(r_y)^{l.alg})$$
\end{prop}
\begin{proof} The proof follows the proof of the Proposition 4.34 \cite{MR3529394} by replacing $\lambda$ with $\sigma_{\mathcal{P}}$ everywhere.
\end{proof} 

\begin{lemma}\label{4.28}
Let $x$, $y \in \Spm R_{\infty} (\sigma_{min})[1/p]$ be smooth points such that the monodromy operators of $WD(r_x)$ and $WD(r_y)$ are the same and $WD(r_x)|I_F \simeq WD(r_y)|I_F$. Then $\pi_{sm}(r_x)|K \simeq \pi_{sm}(r_y)|K$.
\end{lemma}

\begin{proof} Since $x$ and $y$ are both smooth points, the representations $\pi_{sm}(r_x)$ and $\pi_{sm}(r_y)$ are both irreducible and generic. Moreover, it follows from hypotheses that $\pi_{sm}(r_x)$ and $\pi_{sm}(r_y)$ have the same inertial support and as well as the same number and the same size of segments for Bernstein-Zelevisky classification. So if $\pi_{sm}(r_x)=Q(\Delta_{1}) \times \ldots \times Q(\Delta_{r})$ then there are unramified characters $\chi_{i}$ such that $\pi_{sm}(r_y)=Q(\Delta_{1} \otimes \chi_{1})\times \ldots\times Q(\Delta_{r}\otimes \chi_{r})$. Restricting to $K$, we get $\pi_{sm}(r_x)|K \simeq \pi_{sm}(r_y)|K$.
\end{proof}

\begin{lemma}\label{4.29}
Let $x \in \Spm R_{\infty}(\sigma_{min})[1/p]$ be such that $\pi_{sm}(r_x)$ is generic, then $x \in \mathrm{Supp}\: M_{\infty}(\sigma_{\mathcal{P}_x})$
\end{lemma}

\begin{proof} By Lemma \ref{4.25}, $\pi_{sm}(r_x)$ is a subquotient of $\pi_{x}$, then for any partition-valued function $\mathcal{P}$, we have:
$$\dim_E \mathrm{Hom}_{K}((\sigma_{\mathcal{P}})_{sm}, \pi_{sm}(r_x)) \leq \dim_E \mathrm{Hom}_{K}((\sigma_{\mathcal{P}})_{sm}, \pi_{x})$$

\noindent In particular we have
$$\dim_E \mathrm{Hom}_{K}((\sigma_{\mathcal{P}_x})_{sm}, \pi_{sm}(r_x)) \leq \dim_E M_{\infty}(\sigma_{\mathcal{P}_x}) \otimes_{R_\infty} \kappa(x)$$

\noindent Since $\dim_E \mathrm{Hom}_{K}((\sigma_{\mathcal{P}_x})_{sm}, \pi_{sm}(r_x))\neq 0$ then $\dim_E M_{\infty}(\sigma_{\mathcal{P}_x}) \otimes_{R_\infty} \kappa(x) \neq 0$. This means that $x \in \mathrm{Supp}\: M_{\infty}(\sigma_{\mathcal{P}_x})$.
\end{proof}

\begin{prop}\label{4.30}
Let $x$ be any point of $\Spm R_{\infty}(\sigma_{min})[1/p]$. Then $x \in \mathrm{Supp} \: M_{\infty}(\sigma_{\mathcal{P}}^{\circ})$ implies that $\mathcal{P}_{x} \geq \mathcal{P}$.
\end{prop}

\begin{proof} By Lemma \ref{4.7}, the action of $R_{\infty}$ on $M_{\infty}(\sigma_{\mathcal{P}}^{\circ})$ is a reduced torsion free quotient of $R_{\infty}(\sigma_{\mathcal{P}})'$. So if $x \in \mathrm{Supp} \: M_{\infty}(\sigma_{\mathcal{P}}^{\circ})$ then $x \in \Sp R_{\infty}(\sigma_{\mathcal{P}})' = \Sp R_{\infty}\otimes_{R_{\tilde{\mathfrak{p}}}^{\square}}R_{\tilde{\mathfrak{p}}}^{\square}(\sigma_{\mathcal{P}})$ then by definition of $R_{\tilde{\mathfrak{p}}}^{\square}(\sigma_{\mathcal{P}})$ we have that $\mathcal{P}_{x} \geq \mathcal{P}$.
\end{proof}

\begin{thm}\label{4.31}
Let $x$ a closed $E$-valued point of $\Sp R_{\infty} (\sigma_{min})[1/p]$, such that  $\pi_{sm}(r_{x})$ is generic and irreducible. Then
$$V(r_{x})^{l.alg} \simeq \pi_{sm}(r_x) \otimes \pi_{alg}(r_x)$$
\end{thm}

\begin{proof} By Lemma \ref{4.25} $\pi_{sm}(r_x)$ is a $G$-subquotient of $\pi_{x}$, and for every partition-valued function $\mathcal{P}$
$$\dim_E \mathrm{Hom}_{K}((\sigma_{\mathcal{P}})_{sm}, \pi_{sm}(r_x)) \leq \dim_E \mathrm{Hom}_{K}((\sigma_{\mathcal{P}})_{sm}, \pi_{x})$$

Let $y \in \mathrm{Supp}(M_{\infty}(\lambda)) \cap V(\mathfrak{a})$ a smooth closed point that lies on the same irreducible component $V(\mathfrak{p})$ as $x$. Such a point $y$ always exists because $V(\mathfrak{p})\cap V(\mathfrak{a})\neq 0$ (by definition of an automorphic component). Indeed by Proposition \ref{4.26}, if $y \in V\cap V(\mathfrak{a})$, then $y$ is smooth. Moreover we have that,
$$V(r_{y})^{l.alg} \simeq \pi_{sm}(r_y) \otimes \pi_{alg}(r_y)$$

\noindent by Proposition \ref{4.26}. Then it follows from Proposition \ref{4.27}, that for every partition-valued function $\mathcal{P}$, we have 
$$\dim_E \mathrm{Hom}_{K}((\sigma_{\mathcal{P}})_{sm}, \pi_{x}) =\dim_E \mathrm{Hom}_{K}((\sigma_{\mathcal{P}}), V(r_x)^{l.alg})$$
$$ = \dim_E \mathrm{Hom}_{K}((\sigma_{\mathcal{P}}), V(r_y)^{l.alg}) = \dim_E \mathrm{Hom}_{K}((\sigma_{\mathcal{P}})_{sm}, \pi_{sm}(r_y))$$

\noindent because $x$ and $y$ are both smooth points, lying on the same component. In particular we have that
\[\dim_E M_{\infty}(\sigma_{\mathcal{P}}^{\circ}) \otimes_{R_{\infty}} \kappa(x)= \dim_E M_{\infty}(\sigma_{\mathcal{P}}^{\circ}) \otimes_{R_{\infty}} \kappa(y)\]

\noindent Then  $y \in \mathrm{Supp} \: M_{\infty}(\sigma_{\mathcal{P}}^{\circ})$ if and only if $ x \in \mathrm{Supp} \: M_{\infty}(\sigma_{\mathcal{P}})$.

Taking $\mathcal{P}=\mathcal{P}_x$, by Lemma \ref{4.29} we have that $ x \in \mathrm{Supp} \: M_{\infty}(\sigma_{\mathcal{P}_x}^{\circ})$ hence $y \in \mathrm{Supp} \: M_{\infty}(\sigma_{\mathcal{P}_x}^{\circ})$. Then by Proposition \ref{4.30}, $\mathcal{P}_y \geq \mathcal{P}_x$. Exchanging the roles of $x$ and $y$, we get $\mathcal{P}_y = \mathcal{P}_x$. This means that the monodromy operators of $WD(r_x)$ and $WD(r_y)$ are the same. All together we have:
$$\dim_E \mathrm{Hom}_{K}((\sigma_{\mathcal{P}})_{sm}, \pi_{sm}(r_x)) \leq \dim_E \mathrm{Hom}_{K}((\sigma_{\mathcal{P}})_{sm}, \pi_{sm})=$$
$$= \dim_E \mathrm{Hom}_{K}((\sigma_{\mathcal{P}})_{sm}, \pi_{sm}(r_y))$$

\noindent Similarly using the Proposition 4.34 \cite{MR3529394},we get 
$$\dim_E \mathrm{Hom}_{J}(\lambda_{sm}, \pi_{sm}(r_x)) \leq \dim_E \mathrm{Hom}_{J}(\lambda_{x}, \pi_{x})=$$
$$= \dim_E \mathrm{Hom}_{J}(\lambda_{sm}, \pi_{sm}(r_y))$$

We have shown that $x$ and $y$ have the same monodromy, then by Lemma \ref{4.28}, we get that $\pi_{sm}(r_x)|K \simeq \pi_{sm}(r_y)|K$, so $\pi_{sm}(r_x)|J \simeq \pi_{sm}(r_y)|J$. It follows that the inequality above, is an equality.

We know that the functor $\mathrm{Hom}_{J}(\lambda_{sm}, .)$ is an exact functor. It follows that $\mathrm{Hom}_{J}(\lambda_{sm}, \pi_{sm}(r_x))$ is a subquotient of $\mathrm{Hom}_{J}(\lambda_{sm}, \pi_{x})$ in the category of $\mathcal{H}(G, \lambda_{sm})$-modules, because by Lemma \ref{4.25} $\pi_{sm}(r_x)$ is a subquotient of $\pi_{x}$. Since those two $\mathcal{H}(G, \lambda_{sm})$-modules have the same dimension they must be equal. Using the fact that the functor $\mathrm{Hom}_{J}(\lambda_{sm}, .)$ is an equivalence of categories, we get that $\pi_{sm}(r_x)\simeq \pi_{x}$.
\end{proof}

\begin{coro}\label{4.32}
Let $x \in \Spm R_{\infty}(\sigma_{min})[1/p]$ such that $\pi_{sm}(r_x)$ is generic, then $\mathcal{P}_{x} \geq \mathcal{P}$  implies that  $x \in \mathrm{Supp} \: M_{\infty}(\sigma_{\mathcal{P}}^{\circ})$.
\end{coro}

\begin{proof} From Theorem 3.7 \cite{MR3769675}, it follows that $\mathcal{P}_x$ is the maximal partition $\mathcal{P}$ such that $\mathrm{Hom}_{K}((\sigma_{\mathcal{P}})_{sm}, \pi_{sm}(r_x)) \neq 0$. Then by maximality, $\mathcal{P}_{x} \geq \mathcal{P}$  implies that $\mathrm{Hom}_{K}((\sigma_{\mathcal{P}})_{sm}, \pi_{sm}(r_x)) \neq 0$. However by the Theorem \ref{4.31} combined with Proposition 2.22 of \cite{MR3306557} we have that $\mathrm{Hom}_{K}((\sigma_{\mathcal{P}})_{sm}, \pi_{sm}(r_x)) \neq 0$ if and only if  $M_{\infty}(\sigma_{\mathcal{P}}^{\circ}) \otimes_{R_{\infty}} \kappa(x) \neq 0$. So $x \in \mathrm{Supp} \: M_{\infty}(\sigma_{\mathcal{P}}^{\circ})$.
\end{proof}

\section{Points on automorphic components}\label{A.1}

Recall from section \ref{LA.5}, that 
$$\Sp R_{\infty}(\sigma_{min}) =\mathrm{Supp}(M_{\infty}(\sigma_{min}^{\circ})) = \mathrm{Supp}(M_{\infty}(\lambda^{\circ})) = \Sp R_{\infty}(\lambda).$$

\noindent In what follows we will not differentiate between these four sets.

In this section we will prove that if a Galois representation $r$ is generic and corresponds to a closed point lying on an automorphic component then $BS(r)$ admits a $G$-invariant norm. Let's say a few words about automorphic components. It follows from Proposition \ref{4.9}, that $\Sp R_{\infty}(\sigma_{min})$ is a union of irreducible components of $\Sp R_{\infty}(\sigma_{min})'$. An irreducible component of $\Sp R_{\infty}(\sigma_{min})'$ which is also an irreducible component of $\Sp R_{\infty}(\sigma_{min})$ is called an \textbf{automorphic component}.

By Corollary 2.11 \cite{MR3529394}, we have  $M_{\infty}/\mathfrak{a}M_{\infty} \simeq \tilde{S}_{\xi,\tau}(U^{\mathfrak{p}},\mathcal{O})_{\mathfrak{m}}^{d}$, where the ideal $\mathfrak{a}=(x_1,...,x_h)$  is generated by a regular sequence $(x_1,...,x_h)$ on $M_{\infty}(\sigma_{min}^{\circ})$. We know that $(\varpi,x_1,\ldots, x_h)$ is a system of parameters for $M_{\infty}(\sigma_{min}^{\circ})$.  Then by Lemma 3.9 \cite{MR3544298}, an irreducible component of $\Sp R_{\infty}(\sigma_{min})$ contains a closed point $x \in  \Spm (R_{\infty}(\sigma_{min})/\mathfrak{a})[1/p]$. The set $ \Spm (R_{\infty}(\sigma_{min})/\mathfrak{a})[1/p]$ is finite, since the ring $(R_{\infty}(\sigma_{min})/\mathfrak{a})[1/p]$ is zero-dimensional. This point $x \in \mathrm{Supp}(M_{\infty}(\sigma_{min}^{\circ})) \cap V(\mathfrak{a})$ corresponds to a Galois representation attached to algebraic automorphic forms (cf. section 2.6 \cite{MR3529394}).

Let's outline, briefly, how $x$ gives rise to a Galois representation. By Proposition 5.3.2 \cite{MR3134019} there is a unique lift of a globalization to the Hecke algebra; then by universal property of a global deformation ring we get a surjective map from this global deformation ring to the Hecke algebra. The point $x$ corresponds to a maximal ideal of this Hecke algebra. Thus $x$ corresponds to a maximal ideal of this global deformation ring via the map from global deformation ring to the Hecke algebra. The maximal ideals of global deformation ring correspond to Galois representations of a number field, restricting it to the decomposition group we get a local Galois representation we have been looking for.
 
The components of $\Sp R_{\infty}(\sigma_{min})'$ which contain such a point are precisely the automorphic components. This observation justifies why those components are called automorphic. Now we will deduce some new cases of the Breuil-Schenider conjecture.

\begin{thm}\label{3.1}
Suppose $p\nmid 2n$, and that $r : G_F \longrightarrow GL_{n}(E)$ is a potentially semi-stable Galois representation of regular weight, such that $\pi_{sm}(r)$ is generic. If $r$ corresponds to a closed point $x \in \Sp R_{\infty}(\sigma_{min})[1/p]$, then $\pi_{sm}(r)\otimes \pi_{alg}(r)$ admits a non-zero unitary admissible Banach completion.
\end{thm}

\begin{proof} By Theorem \ref{4.31}, we have that $\pi_{sm}(r)\otimes \pi_{alg}(r) \simeq V(r)^{l.alg}$, and by Proposition 2.13\cite{MR3529394}, $V(r)$ is an admissible unitary Banach space representation, hence a $G$-invariant norm on $V(r)$ restricts to a $G$-invariant norm on $\pi_{sm}(r)\otimes \pi_{alg}(r)$.
\end{proof}

\noindent \textbf{Remark.} It is expected that $\Sp R_{\infty}(\sigma_{min})[1/p] = \Sp R_{\infty}(\sigma_{min})'[1/p]$, i.e. that all the components are automorphic.

\subsection*{Acknowledgments}   The results of this paper are the main part of the author's PhD thesis. The author tremendously grateful to his advisor Vytautas Pa\v{s}k\={u}nas for sharing his ideas with the author and for many helpful discussions. The author would also like to thank the referee for useful comments and corrections. This work was supported by SFB/TR 45 of the DFG.

\bibliographystyle{alpha}
\addcontentsline{toc}{section}{References}
\bibliography{Locally_alg_1}
\nocite{*}

\noindent Morningside Center of Mathematics, No.55 Zhongguancun Donglu, Academy of Mathematics and Systems Science, Beijing , Haidian District, 100190 China
\\
\textit{E-mail address}: pyvovarov@amss.ac.cn
\end{document}